\documentclass[11pt]{article} 
\usepackage[english]{babel}
\usepackage[T1]{fontenc}
\usepackage{lmodern}
\usepackage{fancyvrb} 
\usepackage{graphicx} 
\usepackage{booktabs} 
\usepackage{tabulary} 
\usepackage[utf8]{inputenc} 
\usepackage{blindtext}
\usepackage{hyperref}
\usepackage{xcolor} 
\usepackage{cite} 
\usepackage{mathrsfs}
\usepackage{tikz} 
\usetikzlibrary{arrows,calc,automata,shapes.arrows,positioning} 
\usepackage{fullpage}
\usepackage{framed}
\usepackage[normalem]{ulem}
\usepackage{amsmath}
\usepackage{amsthm}
\usepackage{amssymb}
\usepackage{amsfonts}
\usepackage{fullpage}
\usepackage{cleveref}
\usepackage{microtype}
\usepackage{url}
\usepackage{tabu}
\usepackage{diagbox}
\usepackage{multirow}
\usepackage{makecell}
\usepackage{todonotes}
\usepackage{subcaption}
\usepackage{verbatim}
\usepackage{wasysym}
\usepackage{caption}
\usepackage{tabularx}
\usepackage{array}
\usepackage{enumitem}
\usepackage{setspace}

\usepackage[linesnumbered,ruled,vlined]{algorithm2e}

\SetCommentSty{mycommfont}

\SetKwInput{KwInput}{Input}                
\SetKwInput{KwOutput}{Output}              

\newenvironment{claim}[1]{\par\noindent\underline{Claim:}\space#1}{}

\usepackage{tikz}
\usetikzlibrary{calc}
\tikzstyle{noeud}=[circle,inner sep=2, minimum size =3 pt, line width = 1pt, draw=black, fill=white]
\definecolor{bleu}{rgb}{0.38, 0.31, 0.86}
\definecolor{rouge}{RGB}{162, 27, 43}

\makeatletter

\newskip\@bigflushglue \@bigflushglue = -100pt plus 1fil

\def\bigcentering{\let\\\@centercr\rightskip\@bigflushglue%
\leftskip\@bigflushglue
\parindent\z@\parfillskip\z@skip}

\makeatother

\usepackage{zref-savepos}
\newcounter{NoTableEntry}
\renewcommand*{\theNoTableEntry}{NTE-\the\value{NoTableEntry}}

\DeclareMathOperator{\ex}{ex} 
\DeclareMathOperator{\bal}{bal} 
\DeclareMathOperator{\stab}{Stab} 

\DeclareMathOperator{\len}{len} 
\newtheorem{theorem}{Theorem}
\newtheorem{lemma}[theorem]{Lemma}
\newtheorem{remark}[theorem]{Remark}
\newtheorem{definition}[theorem]{Definition}
\newtheorem{proposition}[theorem]{Proposition}
\newtheorem{prop}[theorem]{Proposition}

\newtheorem{problem}[theorem]{Problem}

\newcommand{\cH}{\mathcal{H}}
\newcommand{\fer}{{\rm Fer}}
 
\newcommand{\LFR}[1]{\left\lfloor #1 \right\rfloor}
\newcommand{\lfr}[1]{\lfloor #1 \rfloor}
\newcommand{\LCR}[1]{\left\lceil #1 \right\rceil}
\newcommand{\lcr}[1]{\lceil #1 \rceil}

\def\restrict#1{\raise-.5ex\hbox{\ensuremath|}_{#1}}

\definecolor{lightblue}{RGB}{51,153,255} 
\definecolor{lime}{RGB}{168,181,49}
\definecolor{greenyellow}{RGB}{178,255,46}
\definecolor{purple}{RGB}{181,82,255}
\definecolor{salmon}{RGB}{250, 128, 11}

\newcommand{\Z}{\mbox{${\mathbb Z}$}}

\title{New recursive constructions of amoebas and their balancing number\thanks{This work has been supported by PAPIIT IG100822.}}

\author{
	Laura Eslava$^\S$ 
	\and Adriana Hansberg$^\ddagger$ 
        \and Tonatiuh Matos Wiederhold$^\mathparagraph$ 
        \and Denae Ventura$^{*}$ 
	\\ \\ \\
	$^\ddagger$ Instituto de Matem\'aticas, UNAM Juriquilla, 76230 Quer\'etaro, Mexico.
	\\
	$^\S$ IIMAS, UNAM Ciudad Universitaria, 04510 Mexico City, Mexico.
	\\
	$^\mathparagraph$ Dept. of Mathematics, University of Toronto, Toronto, Canada.
        \\
	$*$ Dept. of Mathematics, University of California at Davis, USA.}
\date{}

\begin{document}

\maketitle

\begin{abstract}
The definition of amoeba graphs is based on iterative \emph{feasible edge-replacements}, where, at each step, an edge from the graph is removed and placed in an available spot in a way that the resulting graph is isomorphic to the original graph. Broadly speaking, amoebas are graphs that, by means of a chain of feasible edge-replacements, can be transformed into any other copy of itself on a given vertex set (which is defined according to whether these are local or global amoebas). Global amoebas were born as examples of \emph{balanceable} graphs, which are graphs that appear with half of their edges in each color in any $2$-edge coloring of a large enough complete graph with a sufficient amount of edges in each color. The least amount of edges required in each color is called the \emph{balancing number} of $G$.

In a work by Caro et al. [Electronic Journal of Combinatorics, 30(3) P3.9 (2023)], by means of a Fibonacci-type recursion, an infinite family of global amoeba trees with arbitrarily large maximum degree is presented, and the question if they were also local amoebas is raised. In this paper, we provide a recursive construction to generate very diverse infinite families of local and global amoebas, by which not only this question is answered positively, it also yields an efficient algorithm that, given any copy of the graph on the same vertex set, provides a chain of feasible edge-replacements that one can perform in order to move the graph into the aimed copy. All results are illustrated by applying them to three different families of local amoebas, including the Fibonacci-type trees.

Concerning the balancing number of a global amoeba $G$, we are able to express it in terms of the extremal number of a class of subgraphs of $G$. By means of this, we give a general lower bound for the balancing number of a global amoeba $G$, and we provide linear (in terms of order) lower and upper bounds for the balancing number of our three case studies.

\end{abstract}

\section{Introduction}\label{sec:intro}


Amoeba graphs, or simply amoebas, were first introduced in~\cite{CHM19} by means of a graph theoretical definition. In \cite{CHM20}, for a better understanding of the structure of amoebas, a group theoretical setting is introduced in which they can also be defined. In that work, a distinction is made between two different classes: local amoebas and global amoebas, the latter ones corresponding to the amoebas defined in \cite{CHM19}. The property that makes amoebas work are iterative replacements of edges, where, at each step, some edge is substituted by another such that an isomorphic copy of the graph is created. We call such edge substitutions feasible edge-replacements. A global amoeba $G$ is a graph, such that, for $n$ large enough, any copy of $G$ embedded in $K_n$, the complete graph on $n$ vertices, can be moved to any other copy of $G$ in $K_n$ by means of a chain of feasible edge-replacements. A local amoeba is defined analogously with the difference that it is a spanning subgraph of $K_n$. It was shown in \cite{CHM20} that a graph $G$ is a global amoeba if and only if $G \cup K_1$ is a local amoeba. Hence, in the case of global amoebas, we can work with the complete graph $K_{n(G)+1}$. 

Amoebas are interesting because, due to their nice edge-replacement feature, they have certain interpolation properties that can be used in the context of unavoidable patterns in $2$-colorings of the complete graph, as is shown in \cite{CHM19}. Local amoebas played an important role in the search for zero-sum spanning subgraphs in \cite{CHLZ}. They are a special case of closed families, which are also related to bases of matroids, see also \cite{CHLZ}. Concerning global amoebas, it has been shown that they are balanceable and that bipartite global amoebas are omnitonal. A graph $G$ is \emph{balanceable} if, for $n$ large enough, there is an integer $k$ such that every $2$-edge coloring of a complete graph $K_n$ with more than $k$ edges in each color contains a colored copy of $G$ with half (ceiling or floor) of its edges in each color. The smallest such $k$ is called the \emph{balancing number} of $G$, denoted as $\bal(n,G)$. The definition of omnitonal graph is similar, with the difference that one can guarantee the existence of copies of the graph $G$ in every "tonal variation", meaning that there will be copies of $G$ with $r$ red edges and $b$ blue edges, for every pair $r, b \ge 0$ whose sum is $e(G)$, the number of edges of $G$. The balancing number is a parameter that is closely related to the \emph{extremal number} of graphs, denoted as $\ex(n,G)$ and defined as the maximum number of edges that a graph of order $n$ can have which does not contain $G$ as a subgraph.

Even though global amoebas have been determined to be balanceable \cite{CHM19}, nothing about their balancing number has been studied. In this work, we provide an equivalent expression for the balancing number of a global amoeba $G$ in terms of the extremal number of a class of subgraphs of $G$ which, in turn, is related to the partial sums of the degree sequence of $G$; see Theorems~\ref{thm:bal=ex} and \ref{thm:bal_lowbound}. 


In \cite{LGM23}, the authors develop an algorithm for the recognition of global and local amoebas and another one just for global and local amoeba trees. The bottleneck of the algorithm is a subroutine that checks for isomorphisms between graphs, for which there is a quasi-polynomial algorithm announced by Babai in \cite{Babai15} (see also \cite{DoBaHe17}). By means of that, the algorithm that recognizes amoebas can work in quasi-polynomial time as well. In the case of trees, it is known that the isomorphism problem is linear with respect to the number of vertices \cite{AhHoUl74, Ke57}, so the amoebas recognition-algorithm works here in polynomial time. Another interesting algorithmic matter that concerns this paper is the following question. 

\begin{problem}\label{problem}
Given a graph $G$ embedded in $K_n$ that is a global or local amoeba (where $n > n(G)$ in the case of global amoeba, and $n = n(G)$ in the case of local amoeba), and any copy $G'$ of $G$ in $K_n$, determine a chain of feasible edge-replacements that moves $G$ onto $G'$. 
\end{problem}

We will discuss this matter in Section \ref{sec:algorithm}, where we will present a quadratic-time algorithm (in terms of the number of vertices) that gives a solution to the above problem for families of local amoebas constructed via a recursion that is also presented in this work. Indeed, in the previous section (Section \ref{sec:recursive-loc-am}), we propose a general method of constructing infinite families of local amoebas by means of a recursion, see \Cref{dfn:series} and \Cref{thm:recursion}. Concerning this, various recursive constructions of global amoebas have been defined in the literature.  In \cite{CHM20}, the authors construct a family of Fibonacci-type trees $\mathcal{T}$ which they prove to be global amoebas. Other recursive constructions of families of global amoebas can be found in \cite{E2020, CHM21}, where the first reference includes the definition of the family $\mathcal{B}$ that is studied in this work and which is closely related to the family $\mathcal{A}$, with which we work here as well. We note in passing that constructions that yield amoeba trees or certain family of non-dense amoebas, as the one presented in this work, provide also, together with the fact that local amoebas are closed under complementation, a method for constructing a family of dense local amoebas. 

All throughout this paper, with the aim of illustrating the reader with some interesting examples, we accompany our results applying them to the three aforementioned families $\mathcal{T}, \mathcal{A}$ and $\mathcal{B}$. In particular:
\begin{enumerate}
    \item[(i)] We prove that these families consist of, indeed, local amoebas.
    \item[(ii)] We provide a more explicit, case-dependent expression of the algorithm that solves Problem \ref{problem}.
    \item[(iii)] We obtain upper and lower bounds on their balancing number. 
\end{enumerate}

The family $\mathcal{T}$ consists of trees built via a Fibonacci-type recursion and was originally given in \cite{CHM20} as an example of an infinite family of global amoebas with arbitrarily large maximum degree. With item (i), we solve the problem stated in~\cite{CHM20}, where it was asked if the trees in $\mathcal{T}$ were all local amoebas as well (the first five trees were shown there to be indeed local amoebas).

This work follows the next structure. In \Cref{sec:notation} we state notation and relevant preliminary results.  We also introduce the families $\mathcal{T}, \mathcal{A}$ and $\mathcal{B}$ as our three constructions under study. In \Cref{sec:recursive-loc-am}, we provide a recursive construction of local amoebas and its application to our case studies. \Cref{sec:balancing_glo_am} deals with the balancing number of global amoebas in general and how these results are applied to provide lower and upper bounds of the balancing number of our case studies. In \Cref{sec:algorithm}, we present an algorithm that solves \Cref{problem} for the family $\mathcal{T}$ and its implementation based on results from \Cref{sec:recursive-loc-am}. Finally, in \Cref{sec:conclusion} we give some concluding remarks and open problems.

\section{Notation and preliminary results}\label{sec:notation}

Let $X$ be a finite set and let $S_X$ be the symmetric group which consists of all permutations of elements of $X$. As usual, $S_n = S_{[n]}$, where $[n]=\{1,2, \cdots , n\}$. We also write $[n,m]=\{n,n+1,\ldots , m\}$ for integers $n<m$. The automorphism group of a graph $G$ is denoted as ${\rm Aut}(G)$, and so any graph $G$ of order $n$ satisfies that ${\rm Aut(G)}\cong S$ for some $S\leqslant S_n$. We use some basic group theory results in $S_n$. For a $k$-cycle $(i_1 i_2 \dots i_k)$ in $S_n$ and an arbitrary $\sigma \in S_n$, we have that  $\sigma (i_1 i_2 \dots i_k) \sigma^{-1}=(\sigma(i_1) \sigma(i_2) \dots \sigma(i_k)).$ For $n\ge 2$, it is a well-known fact that $S_n$ is generated by the set of transpositions $\{(1j):\, j\in \{2,\ldots, n\} \}$.

Let $X$ be a set and let $G$ be a group with identity $e$. A \emph{(left) group action} is a function $\theta : G\times X \to X$ such that for every $g\in G$ and $x\in X$, $gx=\theta(g,x)\in X$ and the following group action axioms hold. For every $g,h\in G$ and $x\in X$, $g(hx)=(gh)x$, and for every $x\in X$, $ex=x$, where $e$ is the neutral element in $G$. In this case, the group $G$ is said to act on the set $X$ (from the left). 

Consider a group $G$ acting on a set $X$. The \emph{orbit} of an element $x$ in $X$ is the set $Gx=\{gx \mid g\in G\}$. For every $x$ in $X$, the \emph{stabilizer subgroup} $G_x$ of $G$ with respect to $x$ is the set of all elements in $G$ that fix $x$. The action of $G$ on $X$ is called \emph{transitive} if, for any $x, y \in X$, there is an element $g\in G$ such that $gx =y$. 

The next definition involves the pasting of two functions with distinct domains.

\begin{definition}\label{dfn:fug}
Let $X, Y, A, B$ be sets and $f:X\to Y$ and $g:A\to B$ be two functions such that $f(x)=g(x)$ for all $x\in X\cap A$, then $f\cup g:X \cup A \to Y\cup B$ is defined as

\[ 
(f\cup g) (x) = 
\begin{cases} 
     f(x), & \text{if } x\in X\\
        g(x), & \text{if } x\in A\setminus X.
   \end{cases}
\]

\end{definition}

Along this work, we will consider graphs $G = G(V,E)$ on $n(G)$ vertices equipped with a labeling on their vertex set $\lambda : V \to X$, which will always be a bijection. We define $v_x = \lambda^{-1}(x)$, for each $x \in X$, and 
  $L_G = \{ij \mid v_i v_j \in E(G) \}$ with no distinction between $ij$ and $ji$. For each $\sigma \in S_X$,  let $G_{\sigma}$ be the copy of $G$ on the same vertex set $V$ defined by $E(G_{\sigma}) =  \{ v_iv_j  \;|\; \sigma(i)\sigma(j)\in L_G \}$. Notice that each labeled copy of $G$ on the vertex set $V$ corresponds to a permutation $\sigma \in S_X$ and vice versa. For every graph $G'$ on $V$ isomorphic to $G$, there are $|{\rm Aut}(G)|$ different copies $G_{\sigma}$ that correspond to $G'$ and, furthermore, the group $A_G = \{\sigma \in S_X \mid G_\sigma = G \}$ is isomorphic to ${\rm Aut}(G)$.  \\

Notice that the set of labels 
\[L_{G_{\sigma}} = \{\sigma (i) \sigma (j) \mid v_i v_j \in E(G) \}\] 
on the edges of $G_\sigma$ is the same for all $\sigma \in S_X$. The corresponding copies of the vertices and edges of $G$ in $G_{\sigma}$ are given by their labels: the copy of vertex $v_i$ of $G$ is the vertex of $G_\sigma$ having label $i$, while the copy of an edge $v_i v_j \in E(G)$ is the edge of $G_{\sigma}$ having label $ij$. \\

When two groups $P$ and $Q$ are isomorphic, we write $P\cong Q$. If two graphs $G$ and $H$ are isomorphic, we also write $G\cong H$. In any case, the context will be clear. Given $e\in E(G)$ and $e' \in E( \overline{G})$, the graph $G-e+e'$ is obtained from $G$ by performing the \emph{edge-replacement} that substitutes $e$ by $e'$. If $G-e+e'$ is a graph isomorphic to $G$, we say that the edge-replacement is \emph{feasible}. We consider also the so-called \emph{neutral edge-replacement} $\emptyset \to \emptyset$ as a feasible edge-replacement,  which is given when no edge is replaced at all. 

Let 
\[R_G = \{ rs\to kl \mid G- v_r v_s + v_k v_l \cong G \} \cup \{\emptyset \to \emptyset\}\]
be the set of all feasible edge-replacements of $G$ given by their labels and let $R_G^* = R_G \setminus \{\emptyset \to \emptyset\}$. We will use sometimes the notation $e \to e' \in R_G$ when we do not require to specify the labels of the vertices involved in the edge-replacement. Notice that $R_{G_{\rho}}=R_{G}$ for any $\rho \in S_X$, because any $e\to e' \in R_G$ also represents a feasible edge-replacement of any copy $G_{\rho}$ with $\rho \in S_X$. 

Moreover, the set $\fer_G(e\to e')$ consists of all permutations of labels that correspond to the feasible edge-replacement $e\to e'$. More precisely, for an edge-replacement $rs\to k\ell \in R_G^*$, a permutation of the labels $\sigma$ is an element of $\fer_G(rs\to k\ell)$ if and only if $G_\sigma\cong G-v_rv_s+v_kv_\ell$. For the neutral edge-replacement, we set $A_G = \fer_G(\emptyset \to \emptyset) \cong {\rm Aut}(G)$.
We denote by $\fer(G)$ the group generated by the permutations associated to all feasible edge-replacements, that is, $\fer(G)$ is generated by the set $$\mathcal E_G=\bigcup_{e\to e'\in R_G}\fer_G(e\to e').$$ The group $\fer(G)$ acts on the set $\{G_\rho\mid \rho\in S_X \}$ by $(\sigma,G_\rho)\mapsto G_{\sigma\rho}$ where $\sigma\in \fer(G)$. Note that employing this action exhibits what happens when a series of edge-replacements, associated to $\sigma$, is applied on a copy $G_{\rho}$ of $G$ which results in $G_{\sigma \rho}$. Being able to go from any copy $G_{\rho}$ to any other copy $G_{\rho'}$ by following a series of feasible edge replacements means that for any $\rho, \rho' \in S_X$, there exists $\sigma \in \fer(G)$ such that $\rho'=\sigma \rho$, meaning that $\fer(G)=S_X$.

As shown in \cite{CHM20}, the original definitions of local and global amoebas can be given by means of the group $\fer(G)$. In words, a graph $G$ on $n$ vertices is a local amoeba if and only if any other copy of $G$ on the same vertex set can be reached, from $G$, by a chain of feasible edge-replacements, or, equivalently if $\fer(G) \cong S_n$. Similarly, $G$ is a global amoeba if and only if any copy of $G \cup K_1$ can reach any other copy on the same vertex set by a chain of feasible edge-replacements.

\begin{definition}\label{dfn:amoeaba}
A graph $G$ of order $n$ is called a \emph{local amoeba} if $\fer(G) \cong S_n$, and it is called a \emph{global amoeba} if $\fer(G\cup K_1) \cong S_{n+1}$.    
\end{definition}

We will need several times the following useful proposition.

\begin{proposition}[\cite{CHM20}]\label{prop:local2global}
If $G$ is a local amoeba with $\delta(G) \in \{0, 1\}$, then $G \cup K_1$ is a local amoeba, and so $G$ is a global amoeba.  
\end{proposition}

\subsection{Three constructions under study}\label{sec:treedfns}

In this section, we introduce the families of trees $\mathcal{T},\mathcal{A}$ and $\mathcal{B}$ which we will analyse throughout the paper. Broadly speaking, these families are defined recursively as follows. The $k$-th tree in any of the families is obtained from connecting two previously defined graphs $H$ and $J$ by one of their maximum degree vertices, say $u$ and $v$, respectively. The choice of $H$ and $J$ changes from one family to another, as we specify shortly after. We note in passing that the recursions are well defined since such trees have either a unique vertex of maximum degree or two vertices of maximum degree that are similar (i.e., there is an automorphism sending the one into the other). 

First, the Fibonacci-type trees $\mathcal{T} = \{T_k \;|\; k \ge 1\}$ are defined the following way. Let $T_1$ and $T_2$ be both isomorphic to $K_2$, while for $k \ge 3$, $T_k= (H \cup J) + uv$ is built from a copy $H$ of $T_{k-1}$ and a copy $J$ of $T_{k-2}$ by adding an edge between a pair of vertices, one of $H$ and one of $J$, having each maximum degree in its respective tree; see \Cref{fig:Tk}. 

\begin{figure}
    \centering
    \includegraphics[width=0.6\textwidth]{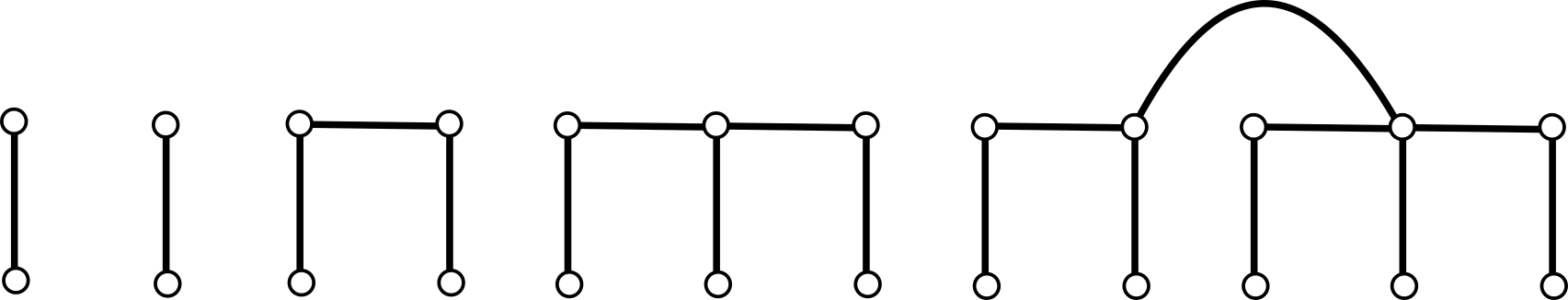}
    \caption{The first five Fibonacci-type trees $T_k$ for $1\leq k \leq 5$.}
    \label{fig:Tk}
\end{figure}

The second family of trees, $\mathcal{A}=\{A_k \mid k\geq 1 \}$, is defined the following way. Let $A_1$ be a single vertex. For $k \ge 2$, $A_k = (H \cup J)+uv$ where $H$ and $J$ are two disjoint copies of $A_{k-1}$, and $u$ and $v$ vertices of maximum degree in $H$ and $J$, respectively; see \Cref{fig:Ak}.

Thirdly, $\mathcal{B}=\{B_k \mid k\geq 1 \}$ may be defined in two equivalent ways. For $k\ge 1$, let $u$ be a vertex of maximum degree in $A_k$ and let $z$ be a new vertex (namely, $z\notin A_k$), then $B_k=A_k + uz$. It is straightforward to verify that this is equivalent to letting $B_1$ be isomorphic to $K_2$ and, for $k\ge 2$, letting $B_k = (H \cup J)+uv$, where $H$ is a copy of $A_{k-1}$ and $J$ is a copy of $B_{k-1}$ ($u$ and $v$ are vertices of maximum degree in $H$ and $J$, respectively); see \Cref{fig:Bk}. 

\begin{figure}
    \centering
    \includegraphics[width=0.5\textwidth]{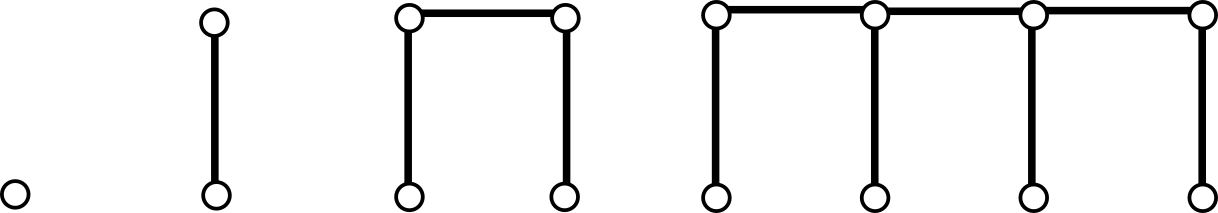}
    \caption{The first four $A_k$ trees for $1\leq k \leq 4$.}
    \label{fig:Ak}
\end{figure}

\begin{figure}
    \centering
    \includegraphics[width=0.5\textwidth]{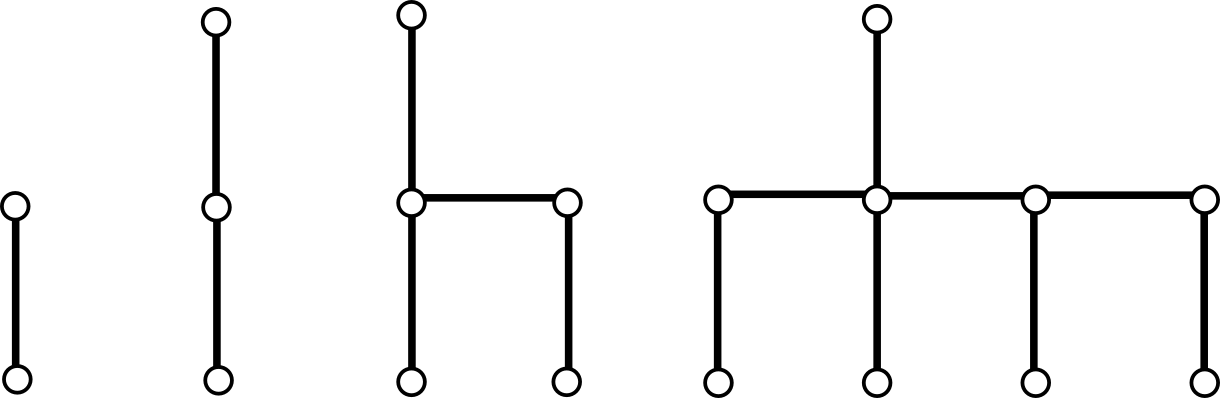}
    \caption{The first four $B_k$ trees for $1\leq k \leq 4$.}
    \label{fig:Bk}
\end{figure}

\section{Recursive local amoebas}\label{sec:recursive-loc-am}

In \cite{CHM20}, different ways of constructing local or global amoebas are presented. Further on, in \cite{CHM21}, the authors work with recursive constructions. In particular, a general method to construct families of global amoebas is developed. Using this method, the authors give an alternative proof of the fact that the Fibonacci-type trees from the family $\mathcal{T}$ are global amoebas. We will deal with this family in this section, too, and we will demonstrate that they are local amoebas as well, solving a problem stated in \cite{CHM20}. This will be achieved by means of a general construction that can be used to generate families of local amoebas recursively. By means of this method, it will also be shown that the families $\mathcal{A}$ and $\mathcal{B}$ consist of local amoebas, too.\\

For a graph $G$ provided with a labeling $\lambda: V(G) \to X$ on its vertices,  consider the set $\mathcal{E}_G^{i}$ of all permutations associated to edge replacements in $R_G$ that fix the label $i \in X$, that is,
$$\mathcal{E}_G^{i}= \mathcal{E}_G \cap \stab_{\fer(G)}(i).$$ 

Let $\fer^{i}(G)$ be the subgroup of $\stab_{\fer(G)}(i)$ generated by the set $\mathcal{E}_G^{i}$. We present the following warm-up lemmas. 

\begin{lemma}\label{lemma:extends}
Let $H$ and $J$ be two vertex disjoint graphs provided with their corresponding disjoint sets of labels $X$ and $Y$. Consider vertices $v_x \in V(H)$, $v_y \in V(J)$ with labels $x \in X$ and $y \in Y$, respectively, and the graph $G = (H \cup J)+v_xv_y$ with the inherited set of labels $X \cup Y$. If $\alpha \in \mathcal{E}^{x}_H$, then $\alpha \cup {\rm id}_{\fer(J)} \in \mathcal{E}^{x}_G$.
\end{lemma}

\begin{proof}
 Given $\alpha \in \mathcal{E}^{x}_H$, then there is a feasible edge replacement $e \to e' \in R_H$ such that $\alpha \in \fer_H(e \to e')$. Since $\alpha (x) = x$, the role $v_x$ is playing in $H$ is the same as in $H_{\alpha}$, and so the edge replacement $e \to e'$ can also be applied on $G$. It follows that 
 \[G_{\alpha \cup {\rm id}_{\fer(J)}} = (H_{\alpha} \cup J) +v_xv_y \cong (H \cup J)+v_xv_y = G,\] 
implying that $e \to e'$ is a feasible edge-replacement in $G$ and $\alpha \cup {\rm id}_{\fer(J)} \in \mathcal{E}^{x}_G$.
\end{proof}

\begin{definition}[Stem-symmetric graph]\label{dfn:series}
Let $G$ be a graph and $v \in V(G)$, and let $\lambda: V(G) \to X$ be a labeling of $G$ and $b = \lambda(v)$. We say that $G$ is \emph{stem-symmetric with respect to $v$} if $\fer^{b} (G)\cong S_{n(G)-1}.$
\end{definition}

The following lemma concerns stem-symmetric graphs $G$. It states that $G$ is, in fact, a local amoeba provided we can exhibit a suitable, additional edge-replacement.

\begin{lemma}\label{lem:stem-sym_local}
    Let $G$ be a labeled graph that is stem-symmetric with respect to a vertex $v$, whose label is $b$. If
    \[\fer(G) \setminus {\rm Stab}_G(b) \neq \emptyset,\] 
    then $G$ is a local amoeba.
\end{lemma}

\begin{proof}
    Let $X$ be the set of labels of $G$. Recall the following group-theoretic result (see, e.g., \cite{R12}). If $x \in X$ and $S \subseteq S_{X \setminus \{x\}}$ is a set of permutations that act transitively on $X \setminus \{x\}$, and if $\varphi \in S_X$ with $\varphi(x) \neq x$, then the set $S \cup \{ \varphi \}$ generates $S_X$. The statement follows by setting $x=b$ and $S = \fer^{b}(G)$.
\end{proof}

\begin{figure}
    \centering
    \includegraphics[width=.7\textwidth]{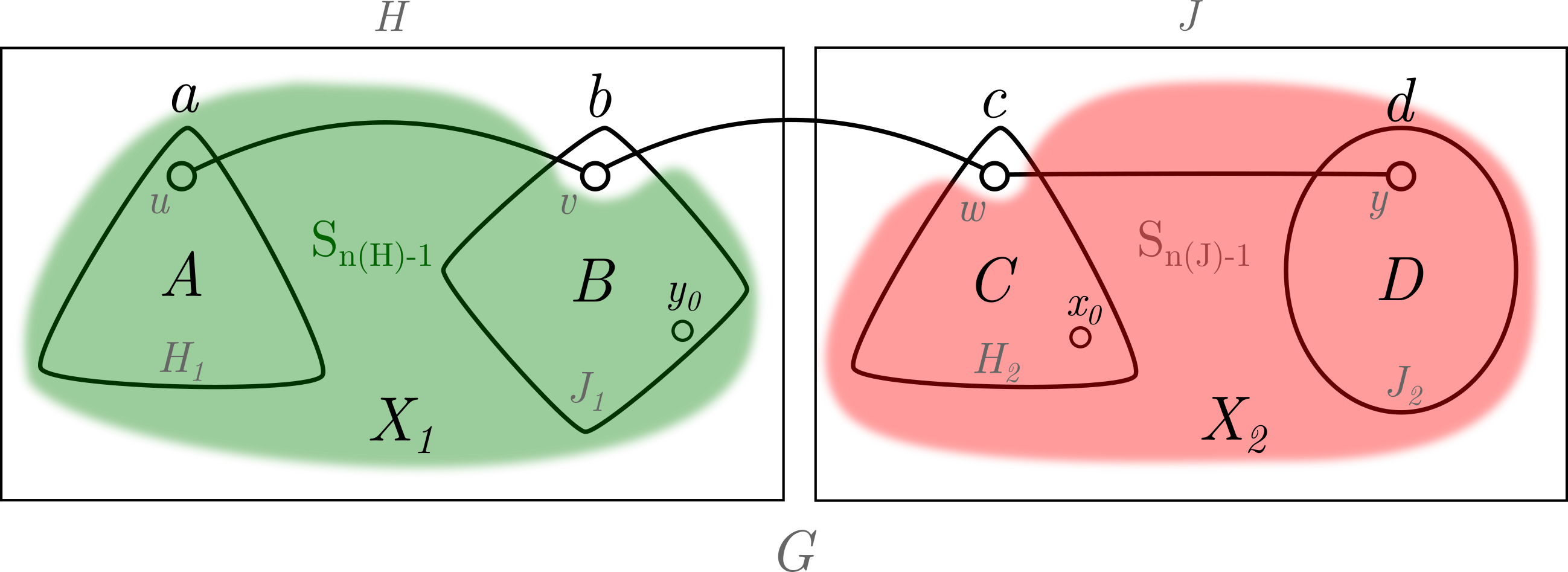}
    \caption{General diagram of \Cref{thm:recursion}.}
    \label{general_diag}
\end{figure}

The theorem below gives a general method by which one can construct local amoebas via a recursion. By \Cref{prop:local2global}, these can be global amoebas, too, if the minimum degree is at most $1$. By means of this theorem, we will be able to prove that the families $\mathcal{T}$, $\mathcal{A}$ and $\mathcal{B}$ consist of local amoebas.

\begin{theorem}\label{thm:recursion}
    Let $H_1, J_1, H_2, J_2$ be vertex disjoint graphs provided with the roots $u, v, w, y$, respectively, and such that $H_1 \cong H_2$ and $u$ is similar to $w$. Let $H = (H_1 \cup J_1) + uv$ be stem-symmetric with respect to $v$, and $J= (H_2 \cup J_2) + wy$ stem-symmetric with respect to $w$. Let $G = (H \cup J) + vw$ be labeled and let $b$ the label on $v$. Then we have the following facts.
    \begin{enumerate}
        \item[(i)] $G$ is stem-symmetric with respect to $v$.
        \item[(ii)] If $\fer(G) \setminus {\rm Stab}_G(b) \neq \emptyset$, then $G$ is a local amoeba.
        \item[(iii)] If $\fer(G) \setminus {\rm Stab}_G(b) \neq \emptyset$ and $\delta(G) \le 1$, then $G$ is a global and a local amoeba.
    \end{enumerate}
    
\end{theorem}

\begin{proof}
\mbox{}\\
(i) We will show first that $G$ is stem-symmetric with respect to $v$.
Let $\lambda(u) = a$, $\lambda(v) = b$, $\lambda(w) = c$, and $\lambda(y) = d$. Let $\lambda(V(H_1)) = A$, $\lambda(V(J_1)) = B$, $\lambda(V(H_2)) = C$, and $\lambda(V(J_2)) = D$, and $X = X_1 \cup X_2$ (clearly, $X_1 = A \cup B$, and $X_2 = C \cup D$). See \Cref{general_diag} for a general diagram of the proof. We take a fixed element $y_0 \in B \setminus \{b\}$ and we will prove that $(x \; y_0) \in \fer^b(G)$ for every $x \in X \setminus \{b\}$, which would yield that $\fer^b(G) \cong S_{n(G)-1}$.

 To this aim, let $\omega : A \to C$ be the bijection induced by an isomorphism from $H_1$ to $H_2$ that sends $u$ to $w$. Then we have $\omega(a) = c$. Note that $cd \to ad$ is a feasible edge-replacement of $G$ that yields the permutation $\rho= \prod_{x \in A} (x\; \omega(x) ) \in \fer_G(cd \to ad)$ which interchanges the elements in $A$ and $C$ by means of  $\omega$ and fixes everything else. In particular, $\rho \in \fer^b(G)$.
 
 Because $\fer^{b}(H) \cong S_{n(H)-1}$, we know that $(x\; y_0)\in \fer^{b}(H)$ for any $x\in (A \cup B) \setminus \{b\}$. Hence, by \Cref{lemma:extends}, $(x\; y_0) \in \fer^{b}(G)$ for any $x\in (A \cup  B) \setminus \{b\}$. 

Take now any $x \in C$ and consider $\omega^{-1}(x) \in A$. The permutations $\rho$ and $(\omega^{-1}(x)\; y_0)$ are contained in $\fer^b(G)$, and thus $(x \; y_0) = \rho (\omega^{-1}(x)\; y_0) \rho^{-1} \in \fer^b(G)$ for any $x \in C$. 

Finally, we consider an arbitrary $x \in D$. Take now a fixed $x_0 \in C \setminus\{c\}$, and observe that $(x \; x_0) \in \fer^{c}(J) \cong S_{n(J)-1}$. By means of \Cref{lemma:extends}, we can even say that $(x \; x_0) \in \fer^{b}(G)$. Since $(x_0 \; y_0) \in \fer^b(G)$, we conclude that $(x \; y_0) = (x \; x_0) (x_0 \; y_0) (x \; x_0) \in \fer^b(G)$.

 Hence, we have considered all possibilities for $x \in X \setminus\{b\}$, and we can assert that $(x \; y_0) \in \fer^b(G)$ for every $x \in X \setminus \{b\}$. Since such a set of permutations generates the symmetric group on $X \setminus \{b\}$, it follows that $\fer^{b}(G) \cong S_{n(G)-1}$, and we are done.

\noindent
 (ii) Since, by (i), $G$ is stem-symmetric with respect to $v$, then, together with the assumption that $\fer(G) \setminus {\rm Stab}_G(b) \neq \emptyset$, $G$ is a local amoeba by \Cref{lem:stem-sym_local}.

 \noindent
 (iii) This follows by (ii) and \Cref{prop:local2global}.
\end{proof}

Observe that the direct application of this method in a recursive manner results in the construction of a family of quite sparse graphs as we are adding each time a single edge connecting two smaller graphs. However, as complements of local amoebas are again local amoebas \cite{CHM20}, it is evident that there are also families of dense local amoebas that can be constructed recursively by these tools, too. 

\subsection{Application to our case studies}\label{subsec:recursive_app}

Theorems~\ref{thm:Tk_Local} and \ref{thm:ABk_local} use \Cref{thm:recursion} to prove that our tree constructions are local amoebas and global amoebas.

\begin{figure}
    \centering
    \includegraphics[width=0.5\textwidth]{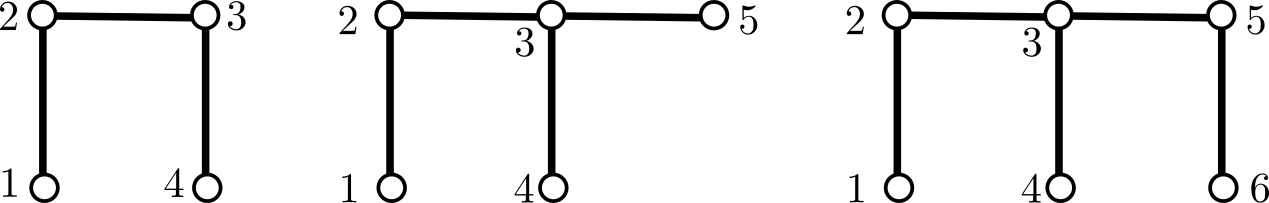}
    \caption{A labeling of $A_3=T_3$, $B_3$ and $T_4$.}
    \label{a3b3t4}
\end{figure}

\begin{theorem}\label{thm:Tk_Local}
For every $k\ge 1$, the Fibonacci-type tree $T_k \in \mathcal{T}$ is both a local and global amoeba.
\end{theorem}

\begin{proof}
Let $k \ge 1$ and let $X$ be the set of labels for $T_k$. Let $v_i$ be the vertex of $T_k$ with label $i \in X$ and let $v_b$ be a vertex of maximum degree in $T_k$. We use induction on $k$ and \Cref{thm:recursion}(i) to prove that $T_k$ is stem-symmetric with respect to $v_b$, then employ~\Cref{lem:stem-sym_local} to prove that $T_k$ is a local amoeba.  For $k=1,2$, the result is trivial. For the cases when $k=3,4$, we proceed to prove that $\fer^{b}(T_3) \cong S_{3}$ and $\fer^{b}(T_4) \cong S_{5}$, and conclude that $T_3$ and $T_4$ are local amoebas using \Cref{lem:stem-sym_local}. Let $T_3$ have the labels $\{1,2,3,4\}$ which are arranged as in \Cref{a3b3t4}. We use the permutations $(1\,2)\in \fer^3 _{T_3}(32\to 31)$ and $(2\,4)\in \fer^3 _{T_3}(12\to 14)$ to generate $\fer^3_{T_3}\cong S_3$. Since $(1\,2\,3\,4)\in\fer_{T_3}(43\to 41)  \in \fer(T_3) \setminus \stab_{\fer(T_3)}(3)$, we obtain with \Cref{lem:stem-sym_local} that $\fer(T_3) \cong S_4$. For $k=4$, let $\{1,2,3,4,5,6\}$ be the labels of $T_4$ arranged as in \Cref{a3b3t4}. We use the permutations $(1\,2)\in \fer^3 _{T_4}(32\to 31)$, $(2\,4)\in \fer^3 _{T_4}(12\to 14)$, $(4\,5)\in \fer^3 _{T_4}(65\to 64)$ and $(56)\in \fer^3 _{T_4}(35\to 36)$ to generate $\fer^3 _{T_4}\cong S_5$. Now consider the permutation $(3\,5)(4\,6)\in \fer_{T_4}(23\to 25)  \in \fer(T_4) \setminus \stab_{\fer(T_4)}(3)$. By \Cref{lem:stem-sym_local}, it follows that $ \fer(T_4) \cong S_6$.

Now let $k \ge 5$ and assume that  $T_i$ is stem-symmetric with respect to its vertex of maximum degree, for every $i < k$. By construction, $T_k$ consists of subtrees $H \cong T_{k-1}$ and $J \cong T_{k-2}$ and an edge $v_bv_c$ joining their vertices $v_b \in V(H)$ and $v_c \in V(J)$ of maximum degree. 
 As $H \cong T_{k-1}$, there are subtrees $H_1 \cong T_{k-3}$ and $J_1 \cong T_{k-2}$ of $H$, with $v_a$ and $v_b$ the vertices of maximum degree in $H_1$ and $J_1$, respectively, such that $H = (H_1 \cup J_1) + v_av_b$. Similarly, there are subtrees $H_2 \cong T_{k-3}$ and $J_2 \cong T_{k-4}$ of $J$, with $v_c$ and $v_d$ the vertices of maximum degree in $H_2$ and $J_2$, respectively, such that $J = (H_2 \cup J_2) + v_cv_d$. Clearly, there is an isomorphism from $H_1$ to $H_2$ that sends $v_a$ to $v_c$. Moreover, by the induction hypothesis, $\fer^{b}(H) \cong S_{n(H)-1}$ and $\fer^{c}(J) \cong S_{n(J)-1}$. Hence, we can use \Cref{thm:recursion}(i) to conclude that $T_k$ is stem-symmetric with respect to $v_b$.

 Now consider the feasible edge replacement $ab \to ac$ of $T_k$ 
and take the permutation $\varphi \in \fer_{T_k}(ab \to ac)$ induced by an isomorphism from $J_1 \cong T_{k-2}$ to $J \cong T_{k-2}$ that sends $v_b$ to $v_c$, i.e. we have $\varphi(b) = c \neq b$. It follows by \Cref{thm:recursion}(ii) that $T_k$ is a local amoeba.

 We conclude, by induction, that every Fibonacci-type tree $T_k \in \mathcal{T}$ is a local amoeba for every $k\in \mathbb{Z}^{+}$.
\end{proof}

\begin{figure}
    \centering
    \includegraphics[width=0.8\textwidth]{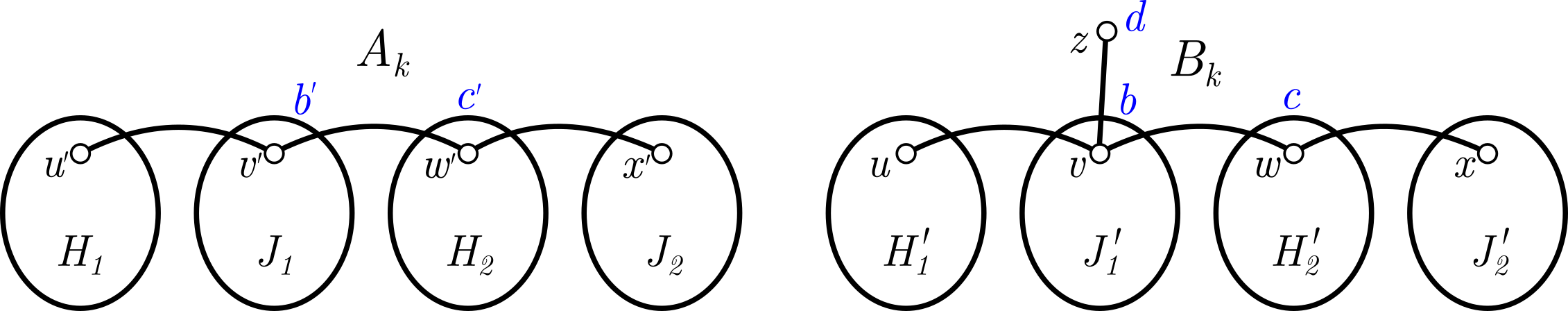}
    \caption{Proof diagrams of \Cref{thm:ABk_local}.}
    \label{fig:akbk}
\end{figure}
\begin{theorem}\label{thm:ABk_local}
For every $k \ge 1$, $A_k \in \mathcal{A}$ and $B_k \in \mathcal{B}$ are both local and global amoebas.
\end{theorem}

\begin{proof}
We proceed in a similar manner as in the proof of \Cref{thm:Tk_Local}. We prove that $A_k$ and $B_k$ are stem-symmetrical with respect to their vertices of maximum degree by means of \Cref{thm:recursion}(i). Afterwards, we use \Cref{thm:recursion}(ii) to prove that $A_k$ and $B_k$ are local amoebas. For $k=1, 2$, the result is given trivially. For $k=3$, note that $A_3\cong T_3$, and by \Cref{thm:Tk_Local}, this case is done. Considering $B_3$, we use the label set $\{1,2,3,4,5\}$ arranged as in \Cref{a3b3t4}. Note that the permutations $(1\,2)\in \fer^3 _{B_3}(32\to 32)$, $(2\,4)\in \fer^3 _{B_3}(12\to 14)$, and $(2\,5)\in \fer^3 _{B_3}(12\to 15)$ generate $\fer^3(B_3)\cong S_4$. As $(4\,5)\in \fer_{B_3}(\emptyset\to \emptyset) \setminus \stab_{\fer(T_3)}(3)$, it follows by \Cref{thm:recursion}(ii) that $B_3$ is a local amoeba.

Now let $k \ge 4$ and assume that, for every $j$ with $1 \le j < k$, $A_j$ and $B_j$ are both stem-symmetric with respect to their vertices of maximum degree as well as local amoebas.

Let $X$ be the set of labels for $A_k$, and $Y$ the set of labels for $B_k$. Let $b \in X$ and $b' \in Y$ be the labels of a vertex of maximum degree in $A_k$ and $B_k$, respectively. We consider first $A_k$. Let $A_k = (H \cup J)+vw$, where both $H$ and $J$ are isomorphic to $A_{k-1}$, and $v$ and $w$ are vertices of maximum degree in $H$ and $J$, respectively. Let $v$, which has maximum degree in $A_k$, have label $b$. Note that $A_k$ can be viewed as four graphs $H_1$, $J_1$, $H_2$, $J_2$, each isomorphic to $A_{k-2}$, and $u \in V(H_1)$, $x \in V(J_2)$ such that $H = (H_1 \cup J_1) + uv$ and $J = (H_2 \cup J_2) + wx$ (see \Cref{fig:akbk}). Let $c$ be the label of $w$. By the induction hypothesis, $\fer^{b}(H) \cong S_{n(H)-1}$, and $\fer^{c}(J) \cong S_{n(J)-1}$. Now \Cref{thm:recursion}(i) implies that $A_k$ is stem-symmetric with respect to $v$.  

The proof that $B_k$ is stem-symmetric with respect to a vertex of maximum degree works similarly. Let $B_k = (H' \cup J')+v'w'$, where $H' \cong B_{k-1}$ and $J' \cong A_{k-1}$, and $v'$ and $w'$ are vertices of maximum degree in $H'$ and $J'$, respectively. Observe first that, by definition, there is a leaf $z$ adjacent to $v'$ such that $H' - z \cong A_{k-1}$ and that $v'$ is the vertex of maximum degree in $B_k$. Now consider graphs $H'_1$, $J'_1$, $H'_2$, $J'_2$ such that $H'_1$, $H'_2$ and $J'_2$ are all isomorphic to $A_{k-2}$ and $J'_1 \cong B_{k-2}$. The vertices $u',v', w', x'$ are vertices of maximum degree in $H'_1, J'_1, H'_2$ and $J'_2$ respectively. Let $H' = (H'_1 \cup J'_1) +u'v'$ and $J' = (H'_2 \cup J'_2)+w'x'$. Let $c'$ and $b'$ be the labels of $w'$ and $v'$, respectively (see \Cref{fig:akbk}). By the induction hypothesis, $\fer^{b'}(H') \cong S_{n(H')-1}$, and $\fer^{c'}(J') \cong S_{n(J')-1}$. Hence, by \Cref{thm:recursion}(i), $B_k$ is stem-symmetric with respect to $v'$. 

Now we show that both $A_k$ and $B_k$ are local amoebas. Consider a permutation $\varphi \in \fer(A_k)$ that interchanges the labels among the sets $V(H)$ and $V(J)$ induced by an isomorphism that sends $b$ to $c$. Since $\varphi(b) = c \neq b$, it follows by \Cref{lem:stem-sym_local} that $A_k$ is a local amoeba. Finally, let $d$ be the label of the leaf $z$ and notice that $db' \to dc'$ is a feasible edge-replacement in $B_k$ that induces the permutation $\varphi' \in \fer(B_k)$ which exchanges the elements between $H'-z$ and $J'$ by means of an isomorphism that takes $v'$ to $w'$. Hence, $\varphi'(b') = c' \neq b'$ and, by \Cref{lem:stem-sym_local}, it follows that $B_k$ is a local amoeba.
\end{proof}

\section{Balancing number of global amoebas}\label{sec:balancing_glo_am}

In this section, we relate the balancing number of a global amoeba $G$ to the Turán number of a class of graphs $\cH_G$  which, broadly speaking, contains all subgraphs of $G$ on half its edges; see \Cref{thm:bal=ex}. This idea was already introduced in \cite{CHLZ} with the aim of finding balanced spanning subgraphs but with respect to closed families and, in particular, of local amoebas. Once in the context of Turán theory, we obtain general lower bounds on the balancing number of a global amoeba $G$, based on the size of cuts of $G$. We then apply these results to our case-studies: we determine the balancing number for the smallest cases and, for the larger ones, we provide lower and upper bounds; the latter is based on estimating the Turán number of suitable star forests, contained in the amoeba, on half the edges of the graph under consideration. 

For any graph $G$, let 
\begin{align}\label{dfn:H_G}
    \cH_G= \left\{ H:\, H\subset G,\, H \text{ has } \LFR{\frac{e(G)}{2}} \text{ edges and no isolated vertices}\right\}
\end{align}
be the class of subgraphs of $G$ that contain half the number of edges of $G$ (rounding down when $e(G)$ is odd). The next theorem establishes the equivalence between the balancing number of a global amoeba $G$ and the Turán number of the class $\cH_G$. We will make use of the following two well-known facts. By a classical argument of Erd\H{o}s \cite{Erd}, every graph contains a bipartite subgraph with at least half its edges. Moreover, $\ex(n,H) = o(n^2)$ for any bipartite graph $H$ (see \cite{FuSi}). Hence, we have that $\ex(n,\cH_G) = o(n^2)$, as $\cH_G$ contains at least a bipartite graph.

\begin{theorem}\label{thm:bal=ex}
Let $G$ be a global amoeba. Then, for $n$ sufficiently large, $\bal(n,G)=\ex(n,\cH_G)$. 
\end{theorem}

\begin{proof}
Throughout the proof, let $m$ be the number of edges of $G$. We show that $\bal(n,G)$ is an upper bound for  $\ex(n,\cH_G)$ and viceversa.  Take $n$ large enough such that all arguments in the proof go through.

To show the inequality $\ex(n,\cH_G) \le \bal(n,G)$, consider a 2-edge-coloring of $K_n$ with partition $E(K_n) = R\cup B$, where $|B| = \ex(n,\cH_G)$ and such that the graph induced by $B$ is free from any member of $\cH_G$. As $\ex(n,\cH_G) = o(n^2)$, it follows that $|R| = \binom{n}{2} - |B| \ge |B|$, for $n$ large enough. Since the graph induced by $B$ is $\cH_G$-free, there is no balanced $G$ in $K_n$, implying that $\ex(n,\cH_G) = |B| \le \bal(n,G)$.

For the other side of the inequality, let $E(K_n) = R \cup B$ be a 2-edge-coloring of $K_n$ satisfying $\min\{|R|, |B|\} > \ex(n,\cH_G)$ (this is possible since $\ex(n,\cH_G) = o(n^2)$). We proceed to show that the coloring contains a balanced copy of $G$. 

Since $|R| > \ex(n,{\cH}_G)$, there is a copy $H_R$ of a member from ${\cH}_G$ contained in the graph induced by $R$. Let $G_R \subset K_n$ be a copy of $G$ such that $H_R\subset G_R$ and let $\ell_R$ be the number of red edges in $G_R$. Similarly, we may 
consider a blue copy $H_B$ of a member from ${\cH}_G$ contained in the graph induced by $B$. Let $G_B \subset K_n$ be a copy of $G$ such that $H_B\subset G_B$ and let $\ell_B$ be the number of blue edges in $G_B$. By the definition of $\cH_G$, we have that $\min\{\ell_R,\ell_B\}\ge \lfr{\frac{m}{2}}$. 

Since our aim is to find a balanced copy of $G$ in $R\cup B$, we may assume that neither $G_R$ nor $G_B$ are balanced, as otherwise we are done; that is, $\{\ell_R,\ell_B\}\cap \{\lfr{\frac{m}{2}},\lcr{\frac{m}{2}}\}=\emptyset$. Therefore, in what follows, we assume that $\min\{\ell_R,\ell_B\}> \lcr{\frac{m}{2}}$. 

Since $G$ is a global amoeba, there exists a sequence of subgraphs $G_0=G_R,G_1,\ldots, G_k=G_B$ of $K_n$ such that, for each $1\le i\le k$, $G_i$ may be obtained from $G_{i-1}$ by a feasible edge replacement. Let $\ell_{B,i}$ be the number of blue edges in $G_i$. Edge replacements may change the number of blue edges in each $G_i$ by at most one, that is, $|\ell_{B,i}-\ell_{B,i-1}|\le 1$. On the other hand, $\ell_{B,0}=m-\ell_R< m-\lcr{\frac{m}{2}} = \lfr{\frac{m}{2}}$ and $\ell_{B,k}=\ell_B> \lcr{\frac{m}{2}}$. Thus, the theorem of the intermediate value implies that there is some $0\le i\le k$ for which $\ell_{B,i}\in \{\lfr{\frac{m}{2}},\lcr{\frac{m}{2}}\}$ and such $G_i$ is a balanced copy of $G$. Hence, $\ex(n,\cH_G)$ is an upper bound for $\bal(n,G)$, as desired. 
Together, these two arguments establish $\bal(n,G)=\ex(n,\cH_G)$.
\end{proof}

Moreover, we have the following lower bound on $\bal(n,G)$ for general, global amoebas. 

\begin{theorem}\label{thm:bal_lowbound}
Let $G=(V,E)$ be a global amoeba with $k$ vertices, $m$ edges and degree sequence $d_1\ge \cdots \ge d_k$. For any $\ell\ge 1$ satisfying 
$\sum_{i=1}^\ell d_i < \LFR{\frac{m}{2}}$, we have $\bal(n,G)\ge \ell (n-\ell)$. 
In particular, 
\begin{align*}
\bal(n,G)\ge \LFR{\frac{m-1}{2k+1}} \left(n-\LFR{\frac{m-1}{2k+1}}\right).
\end{align*}
\end{theorem}

Note that the bound expressed in the theorem is meaningful only when $m\ge 2(k+1)$. The proof of this theorem, which is a consequence of \Cref{thm:bal=ex} and the next two lemmas, requires the introduction of the following notation. For a graph $G=(V,E)$ with partition $V=S\cup T$, let 
\begin{align}\label{dfn:C}
    e_G(S,T)=|\{uv\in E(G):\, u\in S, v\in T\}|;
\end{align}
we say that $e_G(S,T)$ is the size of the cut $(S,T)$. For $1\le \ell\le |V|$, we define 
\begin{align}
    \mathrm{MaxCut}(G,\ell) &=\max \{e_G(S,V\setminus S):\, S\subset V, |S|\le \ell\}; \\
    \ell_G&=\max\left\{\ell \ge 0: \mathrm{MaxCut}(G,\ell)< \LFR{\frac{e(G)}{2}} \right\}. \label{dfn:k_G}
\end{align}

Observe that $0 \le \ell_G < |V|/2$. To see this, note that, as a function of $\ell$, $\mathrm{MaxCut}(G,\ell)$ is non-decreasing
and attains its maximum at $\ell=\lfr{\frac{|V|}{2}}$; while, as mentioned before, every graph $G=(V,E)$ has a cut of size at least $e(G)/2$ \cite{Erd}; that is, $$\mathrm{MaxCut}\left(G,\LFR{\frac{|V|}{2}}\right)\ge \LCR{\frac{e(G)}{2}}.$$ 

The following theorem gives a lower bound on $\ex(n,\cH_G)$ in terms of $\ell_G$.

\begin{lemma}\label{lemma:cut}
Let $G=(V,E)$ be a graph. Then, for any $\ell \le \ell_G$,  
$$\ex(n,\cH_G)\ge \ell_G(n-\ell_G) \ge \ell(n-\ell).$$
\end{lemma}

\begin{proof}
For any $n \ge |V|$, using $0 \le \ell_G < |V|/2$, we infer that the maximum of $\{\ell(n-\ell): 0\le \ell \le \ell_G\}$ is attained at $\ell=\ell_G$. Let us assume $\ell_G \ge 1$; otherwise, the statement is trivially satisfied. We claim that the complete bipartite graph $K_{\ell_G,n-\ell_G}$ does not contain any $H\in \cH_G$ as a subgraph and so $\ex(n,\cH_G)\ge \ell_G(n-\ell_G)$. 

Suppose to the contrary that $K_{\ell_G,n-\ell_G}$ contains a copy of some $H\in  \cH_G$, in which case $H$ is bipartite with partition $S\cup R$ and size $\lfr{\frac{m}{2}}$, where $m = e(G)$. Without loss of generality, we may assume that $|S|\le \ell_G$.

Considering $G$, by definition of $\ell_G$ and \eqref{dfn:k_G}, and assuming $\ell_G \ge 1$,  we have 
\begin{align*}
    e_G(S,V\setminus S)\le \mathrm{MaxCut}(G,\ell_G)< \LFR{\frac{m}{2}}.
\end{align*}
On the other side, since $H$ is bipartite and subgraph of $G$, 
\begin{align*}
    e_G(S, V\setminus S) \ge e_H(S,R) = e(H)= \LFR{\frac{m}{2}};
\end{align*}
which is a contradiction. 
\end{proof}

In preparation for the proof of \Cref{thm:bal_lowbound}, let us consider solutions to the following quadratic inequality.

\begin{lemma}\label{lemma:aux_q}
    Consider $k,m\ge 1$ such that $(2k+1)^2\ge 4(m-1)$ and set $\ell_0=\lfr{\frac{m-1}{2k+1}}$. Any $x\le \ell_0$ satisfies  
    \begin{align*}
    (k+1)x -\frac{x(x+1)}{2}\le \frac{m-1}{2}.
\end{align*}
\end{lemma}

\begin{proof}
Rearranging the terms, we can see that the inequality holds on $(-\infty,x_-)\cup (x_+,\infty)$; for $x_-$ and $x_+$ solutions of the quadratic equation 
\begin{align*}
    2(k+1)x-x(x+1)-(m-1)= -x^2+(2k+1)x-(m-1)=0.
\end{align*}

The solutions to $x^2-bx+c=0$, provided $b^2>4c$, are given by $$x_\pm=\frac{b \pm \sqrt{b^2-4c} }{2}=\frac{b}{2}\left(1\pm \sqrt{1- \frac{4c}{b^2}}\right).$$ 
The function $f(y)=\sqrt{1-y}-1+y/2$ is concave and its maximum is attained at $f(0)=0$; we infer that, $\sqrt{1-y}\le 1-y/2$ for $y\in [0,1]$. This implies that $$x_-=\frac{b}{2}\left(1-\sqrt{1-\frac{4c}{b^2}}\right) \ge \frac{b}{2}\left(1-\left(1-\frac{2c}{b^2}\right) \right)= \frac{c}{b}.$$

Using $b=2k+1$, $c=m-1$ and the definition of $\ell_0$, we obtain that 
$x_-\ge \frac{m-1}{2k-1} \ge \ell_0$.
\end{proof}

\vspace{2ex}
\begin{proof}[Proof of \Cref{thm:bal_lowbound}]
Using that the degree sequence $d_1,d_2,\ldots, d_k$ is listed in non-increasing order, we infer that for any $S\subset V$, 
\begin{align*}
    e_G(S,V\setminus S)\le  
    \sum_{v\in S} \deg(v)\le \sum_{i=1}^{|S|} d_i. 
\end{align*}
It follows that, if $\ell\ge 1$ satisfies $\sum_{i=1}^\ell d_i < \LFR{\frac{m}{2}}$, then $\ell \le \ell_G$ since
\begin{align*}
    \max \{e_G(S,V\setminus S):\, S\subset V, |S|\le \ell\}
    \le \sum_{i=1}^\ell d_i
    < \LFR{\frac{m}{2}}. 
\end{align*}
This, together with \Cref{thm:bal=ex} and \Cref{lemma:cut}, implies that $\bal(n,G)=\ex(n,\cH_G)\ge \ell (n-\ell)$. 

Finally, we optimize the value of $\ell$ such that $\sum_{i=1}^\ell d_i < \LFR{\frac{m}{2}}$. Since $G$ is a global amoeba, by \cite[Proposition 3]{CHM20}, we have $d_i\le k+1-i$ for $1\le i\le k$, then
\begin{align*}
 \sum_{i=1}^\ell d_i\le \sum_{i=1}^\ell (k+1-i)=(k+1)\ell -\frac{\ell(\ell+1)}{2}.   
\end{align*}
Since we are considering integer variables $k,m$ and $\ell$, the inequality 
    \begin{align*}
    (k+1)x -\frac{x(x+1)}{2}<\LFR{\frac{m}{2}}
\end{align*}
is equivalent to that in the statement of Lemma~\ref{lemma:aux_q}, whose conditions are satisfied since $(2k+1)^2>8{k\choose 2}>4m>4(m-1)$. Thus, $\ell_0=\lfr{\frac{m-1}{2k+1}}$ satisfies  $    \sum_{i=1}^\ell d_i <\LFR{\frac{m}{2}}$ and so $\bal(n,G)=\ex(n,\cH_G)\ge \ell_0 (n-\ell_0)$, as desired.
\end{proof}


\subsection{Applications to our case studies}

For a graph $G$ in any of the classes $\mathcal T,\mathcal A$ or $\mathcal B$ of global amoebas, we determine the balancing number $\bal(n,G)$ for the smallest cases (see \Cref{table:cte_bal}), and, for the larger ones, we provide upper and lower bounds. 

We begin by presenting \Cref{thm:TABk_bal}, which provides a summary of the asymptotic results for each class. We then obtain, in each of the subsequent sections, the proofs for the lower and upper bounds in \Cref{thm:TABk_bal}, and finally, we calculate the balancing numbers corresponding to \Cref{table:cte_bal} and the proof of \Cref{thm:TABk_bal}.

\begin{theorem}\label{thm:TABk_bal}
For $n$ sufficiently large, $$n-1\le \bal(n,T_6)\le 2n-2.$$ 
For $k\ge 7$, 
$$F_{k-4}(n-F_{k-4})\le \bal(n,T_k)\le (F_{k-2}-2)n- {F_{k-2}-1\choose 2} +\delta_{k=7};$$
where $\delta_{k=7}=1$ if $k=7$ and is zero otherwise. 
And for $k\geq 5$, $$2^{k-5}(n-2^{k-5})\leq \bal(n,A_k)\leq \bal(n,B_k)  \leq(2^{k-3}-1)n -{2^{k-3}\choose 2}.$$
\end{theorem}

Before continuing, we tackle the relation between the balancing number of $A_k$ and $B_k$ in the following result. 

\begin{lemma}\label{lem:AleqB_bal}
For $k\geq 2$, we have that $\bal(n, A_k)\leq \bal(n, B_k)$.
\end{lemma}

\begin{proof}
Since $B_k$ has $2^{k-1}$ edges, any balanced copy of $B_k$ has $2^{k-2}$ edges of each color. On the other hand, by definiton of $B_k$, there exists one leaf $z$ in $B_k$ such that $B_k- z \cong A_k$. This implies that, by removing one (particular) edge from a balanced copy of $B_k$, we obtain a balanced copy of $A_k$. 
Therefore, any $2$-edge coloring with a balanced copy of $B_k$ contains a balanced copy of $A_k$. This implies $\bal(n, A_k)\leq \bal(n, B_k)$, as desired.
\end{proof}

\begin{table}
    \centering
    \begin{tabular}{ |c|c|c| } 
 \hline
 $G$ & $\bal(n,G)$\\
 \hline
 $T_1=T_2=A_2=B_1\cong K_2$ & 0 \\ 
 $T_3=A_3\cong P_3$, $B_2\cong P_2$, $T_4$, $B_3$ & 1 \\
 $A_4$ & 3 \\
 $T_5$, $B_4$ & 6\\
 \hline
\end{tabular}
\caption{The balancing numbers for the first graphs $G$ in $\mathcal{T}$, $\mathcal{A}$ and $\mathcal{B}$.}
    \label{table:cte_bal}
\end{table}

\subsubsection{Lower bounds}\label{subsec:low_bal}

We begin by defining an equivalent expression for the degree sequence of $G$. For a graph $G$, let $V_i(G)$ denote the set of vertices of degree $i$ in $G$. The profile of $G$ is given by $\{|V_i(G)|; i \ge 0\}$.

Although the general lower bound on \Cref{thm:bal_lowbound} is not meaningful for global amoeba trees, the construction of the classes in our case studies allows us to handle finer bounds for the partial sums $\sum_{i=1}^\ell d_i$ of the degree sequence of the graphs analysed.  

We begin by computing the profile of trees in $\mathcal T$ and $\mathcal A$. For a non-negative integer $k$, let $F_k$ denote the $k$-th Fibonacci number (so that $F_1=F_2=1$, $F_3=2$ and so on).
 
\begin{lemma}\label{lemma:Tk_deg}
The number of vertices in $T_k$ is $2F_k$. For $k\geq 4$, the profile of $T_k$ is given by
 \begin{align*}
     |V_i(T_k)|=
     \begin{cases}
     F_{k+1-i} &1\le i\le 2, \\
     F_{k-1-i} &3\le i\le k-2, \\
     1 &i=k-1, \\
     0 &i\ge k. \\
     \end{cases}
 \end{align*} 
\end{lemma}

\begin{proof}
Recall that,for $k\ge 4$, $T_k$ is obtained from the union of two copies of $T_{k-1}$ and $T_{k-2}$ joined by an edge connecting two vertices of degrees $k-2$ and $k-3$, respectively. In particular,the vertex of maximum degree is unique in $T_k$.

We proceed by induction on $k$. The base cases $k\in \{4,5\}$ can be verified directly (see, also, Figure~\ref{fig:Tk}).
Let $k\ge 6$ and suppose the statement is valid for $k-1$ and $k-2$. 

By the induction hypothesis, $T_{k-1}$ and $T_{k-2}$ together contain precisely two vertices of degree $k-3$ and one vertex of degree $k-2$ (two of which are the maximum degree vertices of their corresponding trees). Thus, $T_k$ contains precisely one vertex of degree $i$ for $i\in\{k-3,k-2,k-1\}$. This establishes $|V_{i}(k)|=1$ for $k-3\le i\le k-1$.

Finally, in constructing $T_k$, the vertices of degree $i$ with $1\le i<k-3$ do not alter their degrees. That is, $V_i(T_k)$ is the disjoint union of $V_i(T_{k-1})$ and $V_i(T_{k-2})$; which implies the recursive equation  
\begin{align*}
    |V_i(T_k)|=|V_i(T_{k-1})|+|V_i(T_{k-2})|;
\end{align*}
completing the proof for the expression of $|V_{i}(k)|$ in both the cases $i\in \{1,2\}$ and $i\in \{3,\ldots, k-4\}$.
\end{proof}

\begin{lemma}\label{lemma:Ak_deg}
The number of vertices in $A_k$ is $2^{k-1}$. For $k\geq 2$, the profile of $A_k$ is given by
 \begin{align*}
     |V_i(A_k)|=
     \begin{cases}
     2^{k-i-1} &1\le i\le k-2, \\
     2 &i=k-1, \\
     0 &i\ge k. \\
     \end{cases}
 \end{align*}     
\end{lemma}

\begin{proof}
Recall that $A_k$ is obtained from the union of two copies of $A_{k-1}$ joined by an edge connecting two vertices of maximum degree; each copy containing two such maximal degree vertices which, additionally, are similar. 

We proceed by induction on $k$. The base case $k=2$ can be verified directly (see, also, Figure~\ref{fig:Bk}).
Let $k\ge 3$ and suppose the statement is valid for $k-1$. 

In two copies of $A_{k-1}$ there are four vertices of degree $k-2$ which are similar. When we join two of them to form $A_k$ we obtain two vertices of degree $k-1$ and are left with two vertices of degree $k-2$. Thus, $|V_i(A_k)|=2$ for $k-2\le i\le k-1$. On the other hand, the vertices of degree $i$ with $1\le i< k-2$ do not alter their degrees. That is $V_i(A_k)$ is the disjoint union of two copies of $V_i(A_{k-1})$; which implies, by the induction hypothesis 
\begin{align*}
    |V_i(A_k)|=2|V_i(A_{k-1})|=2^{k-i-1};
\end{align*}
completing the proof for the cases $1\le i<k-2$.
\end{proof}

We are ready to prove the lower bounds using \Cref{thm:bal_lowbound}.

\begin{prop}\label{prop:Tk_lowbal}
For $k\geq6$, $\bal(n,T_k)\ge F_{k-4}(n-F_{k-4})$.
\end{prop}

\begin{proof}
Recall that $T_k$ has $m=2F_k-1$ edges and note that $\lfr{\frac{m}{2}}=F_k-1$. Let $d_1,d_2, \ldots, d_{2F_k}$ be the degree sequence of $T_k$. We will show that for $k\ge 6$,
\begin{align}\label{eq:boundfib}
    \sum_{i=1}^{F_{k-4}} d_i\le F_k-2< \LFR{\frac{m}{2}},
\end{align}
from which the statement of the lemma follows by \Cref{thm:bal_lowbound} with $\ell=F_{k-4}$.

First, consider the case $k=6$ separately. In this case, $F_{k-4}=1$, $d_1=5$ and $F_k-2=6$, so \eqref{eq:boundfib} is verified directly. 

Now, let $k\ge 7$ and let $V_{>4}(T_k)$ denote the set of vertices in $T_k$ with degree larger than four. By \Cref{lemma:Tk_deg},
\begin{align*}
    |V_{>4}(T_k)|=\sum_{i=5}^{k-1} |V_i(T_k)|= 1+ \sum_{i=5}^{k-2} F_{k-i-1} = 1 +\sum_{i=1}^{k-6} F_i.
\end{align*}
It follows from induction on $m\in \Z_+$, that $\sum_{i=1}^m F_i= F_{m+2}-1$ for $m\ge 1$. Hence $|V_{>4}(T_k)|=F_{k-4}$. Therefore, 
\begin{align}\label{eq:di}
    \sum_{i=1}^{F_{k-4}}d_i= \sum_{v\in V_{>4}(T_k)} \deg_{T_k}(v) = \sum_{i=5}^{k-1} i |V_i(T_k)|.
\end{align}
On the other hand, using the recursion of Fibonacci numbers and that $F_{m}\le 2F_{m-1}$ for $m\ge 3$, we have
\begin{align*}
    2F_{k-2}=4F_{k-4}+2F_{k-5}\le 3F_{k-4}+4F_{k-5}.
\end{align*}
Consequently, together with \Cref{lemma:Tk_deg}, we infer 
\begin{align}\label{eq:F5}
    \sum_{i=1}^{4} i |V_i(T_k)| = F_k+2F_{k-1}+3F_{k-4}+4F_{k-5}\ge F_k+2F_{k-1}+2F_{k-2}=3F_{k}.
\end{align}
Finally, by the hand-shake lemma $\sum_v \deg_{T_k}(v)= \sum_{i=1}^{k-1} i |V_i(T_k)| = 2(2F_k-1)$, which together with \eqref{eq:di} and \eqref{eq:F5} yields
\begin{align*}
\sum_{i=1}^{F_{k-4}}d_i= \sum_{i=5}^{k-1} i |V_i(T_k)|=4F_k-2 -\sum_{i=1}^{4} i |V_i(T_k)| \le F_k-2;
\end{align*}
establishing \eqref{eq:boundfib} for $k\ge 7$ and completing the proof.
\end{proof}

Similarly, we employ \Cref{lemma:Ak_deg} to provide a lower bound of $\bal(n,A_k )$.

\begin{proposition}\label{prop:Ak_lowbal}
For $k\geq 5$, $ \bal(n, A_k)\ge 2^{k-5}(n- 2^{k-5})$.
\end{proposition}

\begin{proof}
Recall that $A_k$ has $m=2^{k-1}-1$ edges and note that $\lfr{\frac{m}{2}}=2^{k-2}-1$. Let $d_1,d_2, \ldots, d_{2^{k-1}}$ be the degree sequence of $A_k$. We will show that for $k\ge 5$,
\begin{align}\label{eq:boundA'}
    \sum_{i=1}^{2^{k-5}} d_i\le 2^{k-2}-2< \LFR{\frac{m}{2}},
\end{align}
from which the statement of the lemma follows by \Cref{thm:bal_lowbound} with $\ell=2^{k-5}$.

First, consider the case $k=5$ separately. In this case, $2^{k-5}=1$, $d_1=4$ and $2^{k-2}-2=6$, so \eqref{eq:boundfib} is verified directly. 

Now, let $k\ge 6$ and 
let $V_{>4}(A_k)$ denote the set of vertices in $A_k$ with degree larger than four. By \Cref{lemma:Ak_deg}, and using the notation within, 
\begin{align*}
    |V_{>4}(A_k)|=\sum_{i=5}^{k-1} |V_i(A_k)|= 2+ \sum_{i=5}^{k-2} 2^{k-i-1} = 2 +\sum_{i=1}^{k-6} 2^i= 2\left(1+\sum_{i=0}^{k-7} 2^i\right) = 2^{k-5}.
\end{align*}
Hence $|V_{>4}(A_k)|=2^{k-5}$. Therefore, 
\begin{align}\label{eq:diA'}
    \sum_{i=1}^{2^{k-5}}d_i= \sum_{v\in V_{>4}(A_k)} \deg_{A_k}(v) = \sum_{i=5}^{k-1} i |V_i(A_k)|.
\end{align}
On the other hand,  together with \Cref{lemma:Tk_deg}, we have
\begin{align}
    \sum_{i=1}^{4} i |V_i(A_k)| &= 2^{k-2}+2\cdot 2^{k-3}+3\cdot 2^{k-4}+4\cdot 2^{k-5} \nonumber\\
    &= 2^{k-5}(4+6+8+8)=2^k-2^{k-2}+2^{k-4}; \label{eq:A'5}
\end{align}
where in the last equality we use that $4+6+8+8=2^5-2^3+2$.

Finally, by the hand-shake lemma $\sum_v \deg_{A_k}(v)= \sum_{i=1}^{k-1} i |V_i(A_k)| = 2(2^{k-1}-1)$, which together with \eqref{eq:diA'} and \eqref{eq:A'5} yields
\begin{align*}
\sum_{i=1}^{2^{k-5}}d_i= \sum_{i=5}^{k-1} i |V_i(A_k)|=2^k-2 -\sum_{i=1}^{4} i |V_i(A_k)| = 2^{k-2}-2^{k-4}-2\le 2^{k-2}-2;
\end{align*}
establishing \eqref{eq:boundA'} for $k\ge 6$ and completing the proof.
\end{proof}

\subsubsection{Upper bounds}\label{subsec:upper_bal}

Throughout this section, for a graph $G$, let $|G|$ denote the number of edges in $G$.
The proof strategy of the next two lemmas is the following. Recall that if $G$ is a global amoeba with $m$ edges, then $\bal(n,G)=\ex(n,\cH_G)$. It then suffices to exhibit $H\in \cH_G$ (that is, $H\subset G$ with $\lfr{\frac{m}{2}}$ edges and no isolated vertices) to find, by \Cref{thm:bal=ex}, the  upper bound
\begin{align*}
\bal(n,G)= \ex(n,\cH_G)\le \ex(H,G).  
\end{align*}

We will thus construct subgraph classes of star forests $\{S_k, k\ge 5\}$ and $\{R_k, k\ge 5\}$ such that $S_k\in \cH_{T_k}$ and $R_k\in \cH_{B_k}$. The choice of star forests allows us to apply the following theorem from \cite[Theorem 3]{LHP13}.

\begin{theorem}[\cite{LHP13}]\label{thm:lidicky}
Let $H=\cup_{i=1}^k S^i$ be a star forest where $d_i$ is the maximum degree of $S^i$ and $d_1\ge d_2\ge \ldots \ge d_k$. For $n$
sufficiently large,
\begin{align*}
    \ex(n,H)=\max_{1\leq i\leq k}\left\{ 
    (i-1)(n-i+1)+{i-1\choose2}+\LFR{\frac{d_i-1}2(n-i+1)} \right\}. 
\end{align*}
\end{theorem}

\begin{prop}\label{prop:Tk_upperbal}
For $n$ sufficiently large, 
$\bal(n,T_6)\le 2n-2$ and, for $k\ge 7$
\begin{align*}
    \bal(n,T_k) \le (F_{k-2}-2)n- {F_{k-2}-1\choose 2} +\delta_{k=7};
\end{align*}
where $\delta_{k=7}=1$ if $k=7$ and is zero otherwise. 
\end{prop}

\begin{proof}
We define a class of star forests $\mathcal{S}=\{S_k | k\ge 5\}$, together with an auxiliarly class $\mathcal{S}^+=\{S^+_k|k\ge 5\}$ such that $S_k,S^+_k\subset T_k$ for $k\ge 5$.

First, let $S_5$ be a star of degree 4 and let $S_6$ be a star forest of one star of degree 4 and one star of degree 3. Then, let $S^+_5$ and $S^+_6$
be obtained, respectively, by adding an independent edge to each of $S_5$ and $S_6$; see Figure~\ref{fig:star_forest}. For $k\ge 7$, let $S_k$ be the disjoint union of two copies of $S_{k-1}$ and $S^+_{k-2}$, while defining $S^+_k$ as the disjoint union of two copies of $S^+_{k-1}$ and $S^+_{k-2}$. That $S_k,S^+_k\subset T_k$ is clear from  Figure~\ref{fig:star_forest} and the construction of $T_k$.

Note that $|S^+_5|=5=F_5$, $|S^+_6|=8=F_6$ and $|S_5|=4=F_5-1$. It can be verified inductively, taking the base cases $S^+_5$ and $S^+_6$, the star forest $S^+_k$ has $F_{k-4}$ stars of maximum degree 4, $F_{k-3}$ stars of maximum degree 3 and $F_{k-3}$ independent edges. Moreover, $|S^+_k|=|S_k|+1=F_k$ for all $k\ge 5$; the difference between $S^+_k$ and $S_k$ being an additional, independent edge. 
 
For $1\le i< F_{k-2}$, let $s_{k,i}$ denote the maximum degree in the $i$-th star of $S_k$, the above argument shows that 
\begin{align*}
    s_{k,i}=\begin{cases}
        4 &1\le i\le F_{k-4},\\
        3 &F_{k-4}< i \le F_{k-3},\\
        1 &F_{k-3}<i\le F_{k-2}-1.
    \end{cases}
\end{align*}
For $1\le i< F_{k-2}$, let 
\begin{align*}
    h_k(i,n)= (i-1)(n-i+1)+{i-1\choose2}+\LFR{\frac{s_{k,i}-1}2(n-i+1)}. 
\end{align*}
By \Cref{thm:lidicky}, $\ex(n, S_k)=\max_{1\le i<F_{k-2}} h_k(i,n)$. Since the statement in \Cref{thm:lidicky} holds for $n$ sufficiently large, we might as well assume that $n\ge 2F_{k-2}-2$. This implies that $(i-1)(n-i+1)$ is increasing in $i$ and so, as a function of $i$, $h_k(i,n)$ is increasing in each of the intervals $[1,F_{k-4}], [F_{k-4}+1, F_{k-3}]$ and $[F_{k-3}+1,F_{n-2}-1]$. In other words, 
\begin{align*}
    \ex(n,S_k) = \max \{h_k(F_{k-4},n), h_k(F_{k-3},n), h_k(F_{k-2}-1,n)\}.
\end{align*}

We then proceed to compare the terms $h_k(i,n)$ for $i\in \{F_{k-4}, F_{k-3}, F_{k-2}-1\}$. We will be concerned, as $n$ is assumed to be large, with the leading term of the following expressions. Letting $\varepsilon_n\in \{0,\frac{1}{2}\}$, we have  
\begin{align}
    h_k(F_{k-4},n)&=\frac{2F_{k-4}+1}{2}(n-F_{k-4}+1) +{F_{k-4}-1\choose 2}+\varepsilon_n \nonumber\\
    &= \frac{2F_{k-4}+1}{2} n -\frac{(F_{k-4}+3)(F_{k-4}-1)}{2}+\varepsilon_n, \label{eq:h1}\\
    h_k(F_{k-3},n)&=F_{k-3}(n-F_{k-3}+1) +{F_{k-3}-1\choose 2} \nonumber\\
    &=F_{k-3}n -\frac{(F_{k-3}+2)(F_{k-3}-1)}{2}, \label{eq:h2}\\
    h_k(F_{k-2}-1,n)&= (F_{k-2}-2)(n-F_{k-2}+2) +{F_{k-2}-2\choose 2} \nonumber\\
    &= (F_{k-2}-2)n-\frac{(F_{k-2}-2)(F_{k-2}-1)}{2}. \label{eq:h3}
\end{align}
We claim that, together with \Cref{thm:lidicky}, these expressions yield
\begin{align}\label{eq:exSk}
    \ex(n,S_k) = \begin{cases}
        h_k(F_{k-3},n) & k=6,\\ 
        h_k(F_{k-2}-1,n)+1 & k=7,\\ 
        h_k(F_{k-2}-1,n) & k\ge 8.
    \end{cases} 
\end{align}
For the case $k\ge 8$, observe that $F_{k-4}+\frac{1}{2}<F_{k-3}<F_{k-2}-2$, implies that  \eqref{eq:exSk} holds, for $n$  sufficiently large that the constant terms are negligible compared to $n$. For the cases $k\in \{6,7\}$, the expressions in \eqref{eq:h1}--\eqref{eq:h3} above are reduced to 
\begin{align*}
    h_6(F_{2},n)&=\frac{3n}{2}+\varepsilon_n, \qquad \;\, \quad \quad
    h_6(F_{3},n)=2n-2,\qquad
    h_6(F_{4}-1,n)=n-1,\\
    h_7(F_{3},n)&=\frac{5(n-1)}{2}+\varepsilon_n, \qquad 
    h_7(F_{4},n)=3n-5,\qquad
    h_7(F_{5}-1,n)=3n-6;
\end{align*}
which verifies the remaining cases of \eqref{eq:exSk} since $h_7(F_{4},n)= h_7(F_{5}-1,n)+1$.
  
Finally, note that the number of edges in $S_k$ is \begin{align*}
    \sum_{i=1}^{F_{k-2}-1} d_i&=4F_{k-4}+3(F_{k-3}-F_{k-4})+(F_{k-2}-F_{k-3}-1)\\
    &=F_{k-4}+2F_{k-3}+F_{k-2}-1\\
    &= F_{k}-1,
\end{align*}
while $F_k$ has $2F_{k}-1$ edges, by \Cref{lemma:Tk_deg}. In other words, $S_k\in \cH_{T_k}$; which implies that  
$\bal(n,T_k)\le \ex(n,S_k)$. Computing the values of \eqref{eq:exSk}, using the expressions in \eqref{eq:h1}-\eqref{eq:h3}, we obtain the desired upper bounds for $\bal(n,k)$ with $n$ sufficiently large.
\end{proof}

\begin{prop}\label{prop:Bk_upperbal}
 For $k\geq 5$ and $n$ sufficiently large, 
 $$\bal(n,B_k)\leq(2^{k-3}-1)n-{2^{k-3}\choose 2}.$$
\end{prop}

\begin{proof}
Let $R_4\subset B_4$ be a star forest comprised of a star of degree 3 and an independent edge; see Figure~\ref{fig:star_forest_Bk}. For $k\ge 5$, let $R_k$ be the disjoint union of two copies of $R_{k-1}$. 
Note that, since $A_4$ also contains a copy of $R_4$, by the construction of $A_k$ and $B_k$, we may assume that $R_k\subset B_k$ for all $k\ge 5$. 

We can verify inductively that, for all $k\ge 4$, $R_k$ contains precisely $2^{k-4}$ stars of degree 3 and $2^{k-4}$ independent edges. Letting $r_{k,i}$ denote the maximum degree in the $i$-th star of $R_k$, we conclude that
\begin{align*}
    r_{k,i}=\begin{cases}
        3 &1\le i \le 2^{k-4},\\
        1 &2^{k-4}<i\le 2^{k-3}.
    \end{cases}
\end{align*}

Similarly to the previous proof, we will show that, if $n\ge 2^{k-2}$ (and sufficiently large),  $\ex(n,R_k)=h'(2^{k-3},n)$ where, 
\begin{align*}
    h'(i,n)= (i-1)(n-i+1)+{i-1\choose2}+\LFR{\frac{r_{k,i}-1}2(n-i+1)}, 
\end{align*}
for $1\le i\le  2^{k-3}$. 
Again, the condition that $n$ is at least as large as twice the number of components in $R_k$ guarantees that $h'(i,n)$ is increasing in each of the intervals $[1,2^{k-4}]$ and$[2 ^{k-4}+1, 2^{k-3}]$. And so, it suffices to compare the terms
\begin{align*}
    h'(2^{k-4},n)&=2^{k-4}(n-2^{k-4}+1) +{2^{k-4}-1\choose 2},\\
    h'(2^{k-3},n)&=(2^{k-3}-1)(n-2^{k-3}+1) +{2^{k-3}\choose 2}.
\end{align*}
For $k\ge 5$, we have $\frac{3}{2}<2^{k-4}<2^{k-3}-1$. Consequently, for $n$ sufficiently large, and together with \Cref{thm:lidicky}, we obtain
\begin{align*}
    \ex(n,R_k) = \max \{h'(1,n), h'(2^{k-4},n), h'(2^{k-3},n)\}= h'(2^{k-3},n).
\end{align*}
This completes the proof as the number of edges in $R_k$ is $$\sum_{i=1}^{2^{k-3}} d_i=3\cdot2^{k-4}+2^{k-4} = 2^{k-2} =\LFR{\frac{2^{k-1}}{2}};$$ 
that is, $R_k\in \cH_{B_k}$ and so
\begin{align*}
    \bal(n,B_k)\le \ex(n,R_k)&= (2 ^{k-3}-1)(n-2^{k-3}+1) +{2^{k-3}-1\choose 2}\\
    &= (2 ^{k-3}-1)n-{2^{k-3}\choose 2}.
\end{align*}
\end{proof}

\begin{figure}
    \centering
    \includegraphics[width=0.4\textwidth]{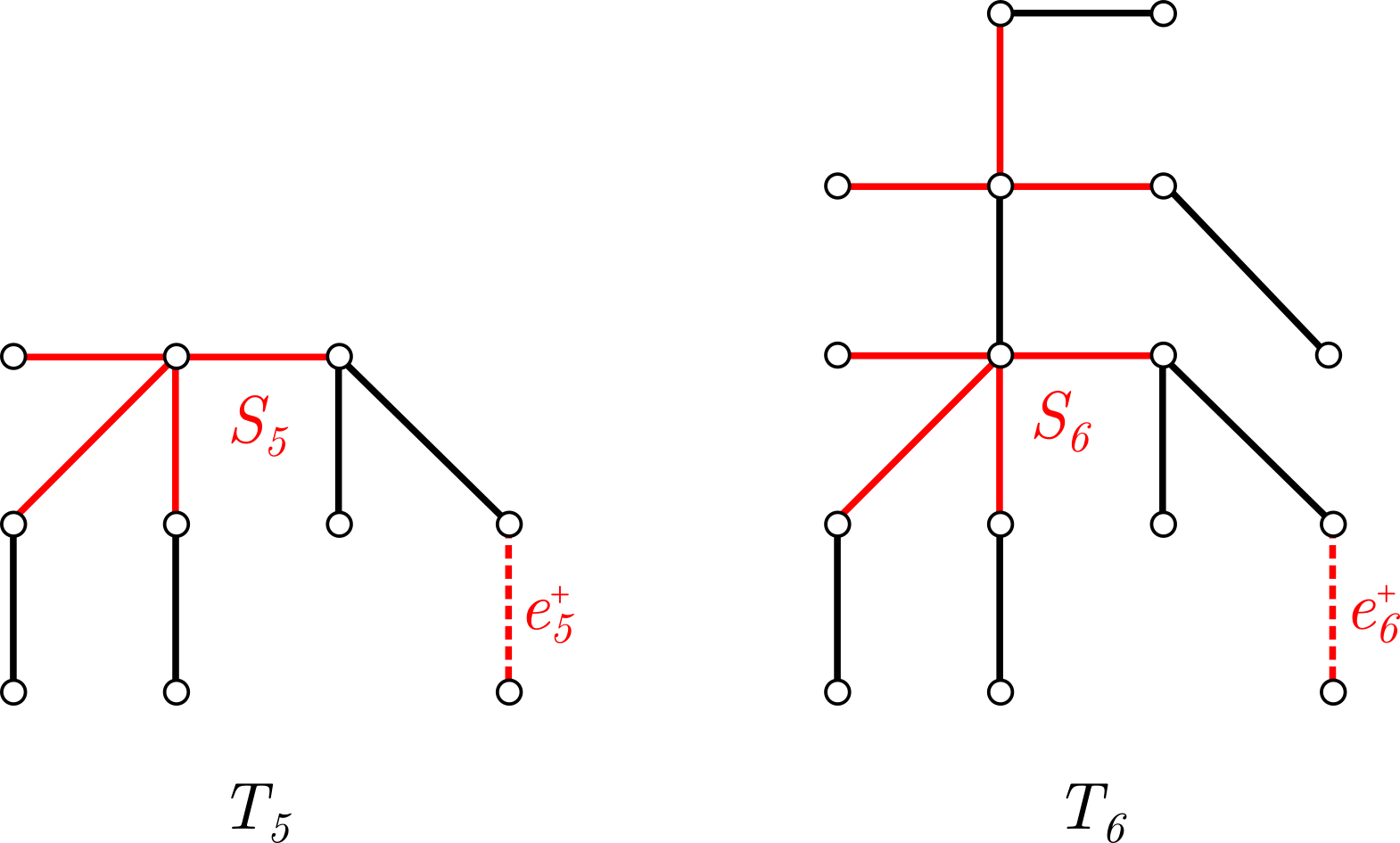}
    \caption{Star forests $S_5$ and $S_6$ in red. The dotted red edges $e_5^+$ and $e_6^+$ define $S_5^+ = S_5 + e_5^+$ and $S_6^+ = S_6 + e_6^+$, respectively.}
    \label{fig:star_forest}
\end{figure}

\begin{figure}
    \centering
    \includegraphics[width=0.4\textwidth]{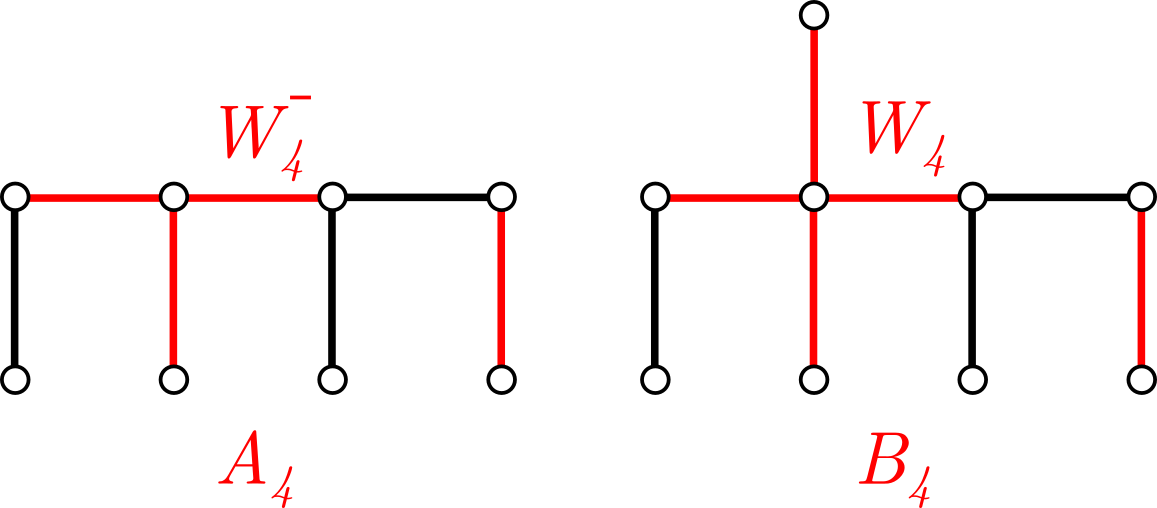}
    \caption{Star forests $W^-_4$ and $W_4$ in red; $W^-_4$ is contained in both $A_4$ and $B_4$.}
    \label{fig:a4b4}
\end{figure}

The proof of \Cref{thm:TABk_bal} is a direct consequence of Propositions \ref{prop:Tk_lowbal} and \ref{prop:Tk_upperbal} for the class $\mathcal{T}$, and \Cref{lem:AleqB_bal}
and Propositions \ref{prop:Ak_lowbal} and \ref{prop:Bk_upperbal} for the classes $\mathcal{A}$ and $\mathcal{B}$.    

\subsubsection{Balancing numbers of the small cases}

Recall that \Cref{thm:bal=ex} provides an equivalence between the balancing number of a global amoeba $G$ and the extremal Turán of $\mathcal{H}_G$ as defined in \eqref{dfn:H_G}. This is straightforward to explore for the first few elements of our case study classes $\mathcal{T}, \mathcal{A}$ and $\mathcal{B}$; which coincidentally, by \Cref{thm:TABk_bal}, are the unique elements of such classes with a constant balancing number.

Within the proof of the next lemma, we make use of the following facts.

 \begin{claim}\label{claim:1}
     For a connected graph $C$ with $m$ edges,
     \begin{enumerate}[label=(\roman*)]
         \item if $m\geq1$, then $C$ contains an edge, $K_2$;
         \item if $m\geq2$, then $C$ contains a path on two edges, $P_2$;
         \item if $m\geq3$, then $C$ is either a triangle or, contains either a star $S_3$ or a copy of $P_{m}$.
     \end{enumerate}
 \end{claim}

\begin{lemma}\label{thm:cte_balancing}
We have the following balancing numbers.
\begin{itemize}
    \item[\it i)] $\bal(n,A_2)=0$, $\bal(n,A_3)=1$ and $\bal(n,A_4)=3$,
    \item[\it ii)] $\bal(n,B_1)=0$, $\bal(n,B_2)=\bal(n,B_3)=1$ and $\bal(n,B_4)=6$,
    \item[\it iii)] $\bal(n,T_1)=\bal(n,T_2)=0$, $\bal(n,T_3)=\bal(n,T_4)=1$ and $\bal(n,T_5)=6.$
\end{itemize}    
\end{lemma}

\begin{proof}
 The first few elements in $\mathcal{T}, \mathcal{A}$
 and $\mathcal{B}$ are paths (or isolated vertices) and their balancing numbers are straightforward to obtain. For the rest of the cases, $T_4$ (which equals $B_3$), $A_4 , T_5$ and $B_4$, we make use of the equivalence in \Cref{thm:bal=ex}. 
 
 Observe that $\mathcal H_{T_4}=\mathcal H_{B_3}$, $\mathcal{H}_{A_4}$ and $\mathcal H_{B_4}=\mathcal H_{T_5}$ consist of all forests on two, three and four edges, respectively; see Figures~\ref{fig:Ha4} and \ref{fig:Ht4}. In particular, $\mathcal H_{T_4}=\mathcal H_{B_3}$ consists of exactly two non-isomorphic graphs: $2K_2$ and $P_2$. Therefore $\ex(n,\mathcal H_{T_4})=\ex(n,\mathcal H_{B_3})=1$, and by Theorem~\ref{thm:bal=ex} we have
 $\bal(n,T_4)=\bal(n,B_3)=1$.

Consider $\mathcal{H}_{A_4}$. To prove that $\ex(n,\mathcal{H}_{A_4})\geq3$, it suffices to consider a triangle as the extremal graph. 
 To prove that $\ex (n,\mathcal{H}_{A_4})\leq3$, let $G$ be a graph on 4 edges with no isolated vertices and connected components $C_1,\dots,C_\ell$ which are in decreasing order in terms of the number of edges. In each of the distinct cases below, we find some $H_i\in \mathcal{H}_{A_4}$.
 
\begin{figure}
    \centering
\includegraphics[scale=0.25]{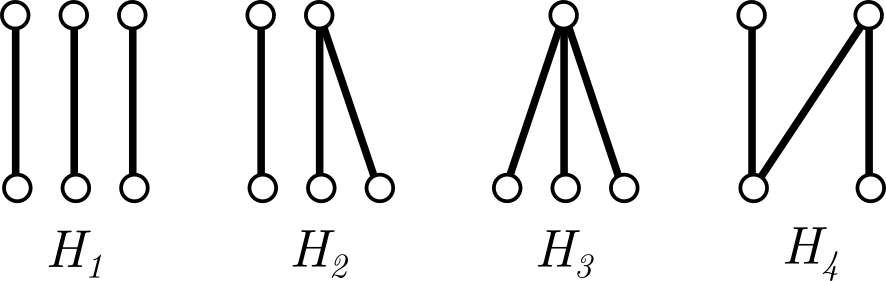}
    \caption{Family $\mathcal{H}_{A_4}$.}
    \label{fig:Ha4}
\end{figure}

When $\ell\geq3$, we may take one edge from each of three components to find $3K_2\cong H_1\subseteq G$.
For the case $\ell=2$, it suffices to apply Claim~(ii) to $C_1$ and Claim~(i) to $C_2$ to produce the subgraph $H_2$, isomorphic to $P_2\cup K_2$. Lastly, if $\ell=1$ ($G$ is connected), we consider the maximum degree $\Delta$. If $\Delta \geq3$, then $G$ contains a copy of the star $H_3$; on the other hand, if $\Delta=2$, since $G$ is connected we may find a path on three edges; that is $H_4\cong P_3$ as a subgraph of $G$. Having covered all possible cases for the components $C_1, \ldots, C_\ell$ we infer $\ex(n,\mathcal{H}_{A_4})=3$, and by \Cref{thm:bal=ex}, $\bal(n,A_4)=3$.
 
 Finally, consider $\mathcal{H}_{T_5}$. We proceed to prove that $\ex(n, \mathcal{H}_{T_5})=6$. 
 To prove the lower bound, it suffices to take a complete graph on four vertices as the extremal graph in both cases. To prove the upper bound, let $G$ be a graph on 7 edges with no isolated vertices and connected components $C_1,\dots,C_\ell$ ordered decreasing in size. In each of the distinct cases below, we find some $F_i\in \mathcal{H}_{T_5}$.

\begin{figure}
    \centering
    \includegraphics[scale=0.25]{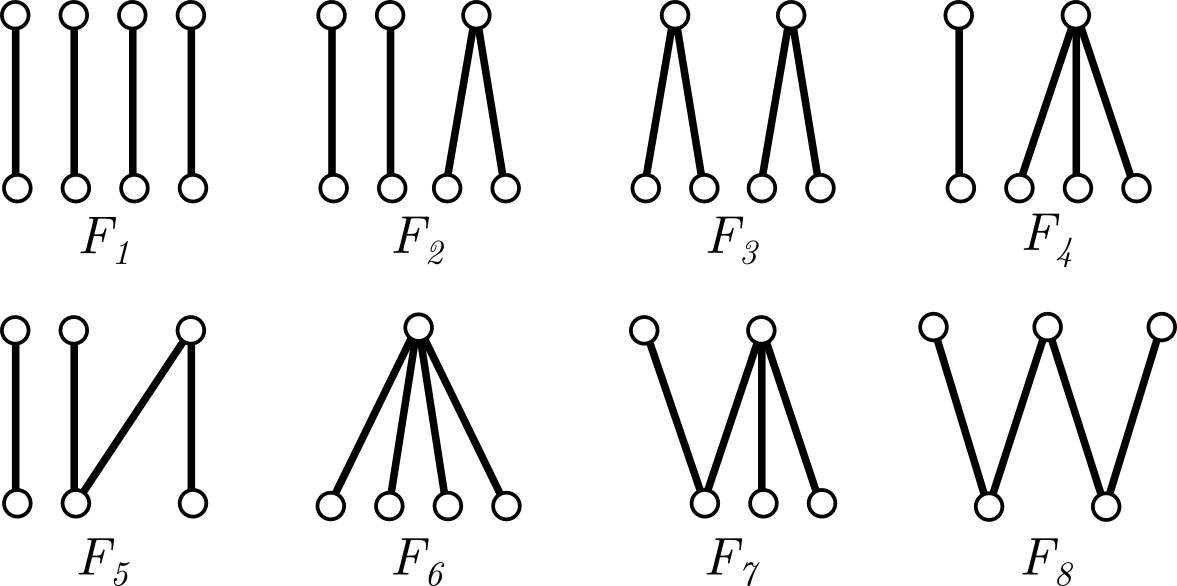}
    \caption{Family $\mathcal{H}_{T_5}=\mathcal{H}_{B_4}$.}
    \label{fig:Ht4}
\end{figure}

 Suppose that $\ell\geq4$; we may take one edge from each of three components to find $3K_2\cong H_1\subseteq G$. If $\ell=3$, then $e(C_1)\geq 2$ and so $C_1$ contains a copy of $P_2$. We then may take one edge from $C_2$ and $C_3$ to find a subgraph $F_2$ isomorphic to $P_2\cup K_2\cup K_2$. In the case when $\ell=2$ both $C_1$ and $C_2$ have at least two edges each, we can find, by means of two uses of Claim~(ii), that $F_3\subseteq G$ which is isomorphic to two disjoint 2-paths. If, on the other hand, $C_1$ has 6 edges and $C_2$ has one, by Claim~(iii), we can find $F_4\cong S_3\cup K_2$ or $F_5\cong P_3\cup K_2$.
Finally, if $\ell=1$, then $G$ is connected and we consider the maximum degree $\Delta$. First, suppose that $\Delta\geq4$, then $G$ contains a star $F_6\cong S_4$. Now consider $\Delta=3$ and let $v$ be a vertex of maximum degree. We claim there is a vertex $w$ such that $d(v,w)=2$. Otherwise, all vertices would be adjacent to $v$, but $\deg(v)=\Delta=3$, so $G$ could not have 7 edges. We conclude the existence of $F_7\subseteq G$, a star $S_3$ with an edge joined to $w$. If $\Delta=2$, then Claim~(iii) yields $F_8\cong P_4\subseteq G$. Therefore, $\ex(n,\mathcal H_{B_4})=\ex(n,\mathcal H_{T_5})=6$, and by \Cref{thm:bal=ex} we obtain that $\bal(n,T_5)=\bal(n,B_4)=6$.
 \end{proof}

\section{Algorithms and implementation}\label{sec:algorithm}

 In this section, we present an algorithm based on Theorem~\ref{thm:Tk_Local} that finds, given any arbitrary permutation $p$ on the vertices of $T_k$, a corresponding sequence of feasible edge replacements in time $\Theta(|V(T_k)|^2)$. Similar algorithms exist for any of the local amoebas constructed in Section~\ref{sec:recursive-loc-am}.

 We begin by defining a class of objects called \texttt{Fer} with two attributes: a sequence of feasible edge replacements and its corresponding permutation. We denote the length of the sequence of the object \texttt{fer} by \texttt{len(fer)}. The product of two \texttt{Fer} objects multiplies their permutations and concatenates the edge replacements, updating the labels. In this implementation, products of permutations (denoted by concatenation) and products of \texttt{Fer} objects (denoted by $*$) are written \textbf{left-to-right}, in contrast with the rest of the paper. This is to be consistent with Python's symbolic math library \texttt{SymPy}. It is enough to produce a hash map that links every permutation in a generating set of the symmetric group to a \texttt{Fer} object to find the \texttt{Fer} object of any permutation.
 
 We assume that the hash maps for $k=1,2,3,4$, i.e.\ the recursive basis of Theorem~\ref{thm:Tk_Local}, have been previously computed and include all of the permutations (not just the generators). Achieving this is out of the scope of this paper, but our implementation \cite{github} performs this step as well.

 We describe each step of the main program, as well as each case of the function \texttt{RecursiveFer}.
\begin{enumerate}
    \item[Step 1:] Write the permutation as a product of disjoint cycles $C=(a_0,a_1,\dots,a_m)$. Then, $C=(y_0,a_0)(y_0,a_1)\cdots(y_0,a_m)(y_0,a_0)$. This step yields a list of transpositions of the form $(y_0,x)$ where $x$ is moved by the permutation and whose product is the permutation.
    
    \item[Step 2:] We use \texttt{RecursiveFer} to find a \texttt{Fer} object for each transposition. It is enough to pass the primitive $x$ as an argument, and this makes it easier to employ memoization; otherwise, we would need to either encode permutations into immutable objects or defining an object hashing. Python's \texttt{functools} library contains a memoization decorator called \texttt{@cache} which saves the values of the function in an internal hash map. Alternatively, we can achieve the same result by updating \texttt{hashMap} and checking, at the beginning of each iteration, if the arguments have been passed before. This costs $O(1)$ time and drastically shortens the execution. For instance, in our experiments, computing \texttt{RecursiveFer}($(0\ 1)\cdots(464\ 465)$, $k=13$) without memoization takes about 293\% longer than with it. We describe each case here, however details can be found in the proof of Theorem~\ref{thm:Tk_Local}.
    \begin{itemize}
        \item[$x=b$:] Notice that $(y_0,x)=\varphi^{-1}(y_0,\varphi^{-1}(x))\varphi$, and that we can find a \texttt{Fer} object for $(\varphi^{-1}(x),y_0)$ since $\varphi^{-1}(x)\in B$.

        \item[$x\in A\cup B$:] Since $A\cup B\cong T_{k-1}$, a simple recursion works for this case.

        \item[$x\in C$:] We may pick $x_0$ as any neighbor of $c$. In the labeling of Theorem~\ref{thm:Tk_Local}, $x_0=c+1$ works, for example. Observe that $(y_0,x)=\rho^{-1}(y_0,\rho^{-1}(x))\rho$, and that we can find a \texttt{Fer} object for $(\rho^{-1}(x),y_0)$, because $\rho^{-1}(x)\in A$.

        \item[$x\in D$:] Notice that $(y_0,x)=(y_0,x_0)(x,x_0)(y_0,x_0)$. We can find a \texttt{Fer} for $(y_0,x_0)$ since $x_0\in C$ in the same $k$. And we can find one for $(x,x_0)$ because $x-c\in B_k\cap(C_{k-2}\cup D_{k-2})$. After finding the \texttt{Fer} object, we need to shift back the labels to $T_k$ by adding $c$ to all of them.
    \end{itemize}
    
    \item[Step 3:] We get a \texttt{Fer} object corresponding to the permutation obtained from multiplying (using the product $*$) all elements in \texttt{transpositions}. Namely, a sought \texttt{Fer} object for the given permutation.
\end{enumerate}

\begin{algorithm}[!ht]
\DontPrintSemicolon
  
  \caption{Factoring a permutation into a list of feasible edge replacements \texttt{Fer} in $T_k$.}
  \label{alg_tk}
  \SetKwFunction{FHash}{RecursiveFer}
  
  \SetKwProg{Fn}{Function}{:}{}
  @cache\tcp*{Memoization decorator for dynamic execution of recursive function}
    \Fn{\FHash{$x$, $k$}}{
    \If{$k\leq4$}{\KwRet hashMap[$(y_0,x)$] \tcp*{\texttt{Fer} is known for the base cases.}}

    Let $a,b,c,d$ be the roots of $T_k$ and $A,B,C,D$ the sets as defined in Theorem~\ref{thm:Tk_Local}.
    
    \If{$x = b$}{
        $\varphi=$ isomorphism $B\to C\cup D$ with \texttt{Fer} object $(ab\to ac)$\;
        \KwRet $\varphi*$RecursiveFer($\varphi^{-1}(x)$, $k$)$*\varphi$}
    \tcc{If the program gets to this line, it must mean that $x\neq b$.}
    
    \If{$x\in A\cup B$}{\KwRet RecursiveFer($x$, $k-1$)}
    
    \If{$x\in C$}{
        $\rho=$ isomorphism $A\to C$ with \texttt{Fer} object $(cd\to ad)$\;
        \KwRet $\rho*$RecursiveFer($\rho^{-1}(x)$, $k$)$*\rho$
    }
    \If{$x\in D$}{
        $x_0=$ a neighbor of the root $c$ in $C$.\;
        $y_0x_0$ = RecursiveFer($x_0$, $k$)\;
        $xx_0$ = RecursiveFer($x-c$, $k-2$)$ + c$\tcp*{$+c$ shifts all labels in the object by $c$ units.}
        \KwRet $y_0x_0*xx_0*y_0x_0$
    }
    }

    \KwInput{$k$, permutation}
    hashMap = $\{\text{generators for }k\leq4\}$ \tcp*{Save known \texttt{Fer} objects in a hash map.}
    
    \tcc{Step 1: Factor permutation into transpositions of the form $(y_0,x)$.}
    transpositions = $\varnothing$\tcp*{Initialize empty list.}
    \For{$C$ cycle of permutation}{
        \For{$x\in C$}{transpositions.add($y_0$,$x$)\tcp*{Add elements to list.}}
        transpositions.add($y_0$,$a_0$)\tcp*{$a_0$ is the first element of $C$.}
    }
    
    \tcc{Step 2: Find \texttt{Fer} for each transposition.}
    \For{$x$ where $(y_0,x)$ in transpositions}{
        xFer = RecursiveFer($x$, $k$)\;
        hashMap[$(y_0,x)$] = xFer\tcp*{Save result in a hash map.}
    }
    \tcc{Step 3: Multiply hashMap values to get a \texttt{Fer} for permutation.}
    $$\text{fer}=\prod_{(y_0,x)\in\text{transpositions}}\text{hashMap}[(y_0,x)]$$
    \KwOutput{fer}
\end{algorithm}

 For the analysis, we care about expressing the time and space complexity in terms of the number of vertices of $T_k$, namely $n(k)=2F_k$. During Step 1, the maximum number of transpositions any permutation gets factored in is achieved by conjugates of $(0\ 1)(2\ 3)(4\ 5)\cdots(n(k)-2\ n(k)-1)$. Thus, with memoization and fixed $k$, the function \texttt{RecursiveFer}($x$,$k$) gets called at most $n(k)-1$ times. It remains to study how long it takes to find a \texttt{Fer} object for each transposition.

 Fix $k>5$. We may compute, beforehand and for all $q\leq k$, the isomorphisms $\varphi$ and $\rho$, the roots $a,b,c,d$ and the sets $A,B,C,D$ of $T_q$. With memoization, $2F_k$ may be computed in time $O(k)$. The roots can be found in constant time since all is needed are smaller values $F_q$. The algorithm only needs the sets to evaluate Booleans of the form $x\in S$. But, picking the appropriate labeling of $T_k$, this would only require comparing $x$ to the endpoints of $S$, namely two of the roots. Hence, this can still be achieved in constant time. The isomorphisms can be generated in linear time and then applied in linear time.
 
 Multiplying two \texttt{Fer} objects takes linear time in the length of the right-most sequence of replacements, since the labels must be updated one by one. Call this time $m(\ell)$, so that $m(\ell)=O(\ell)$. The product of the permutations takes constant time.

 \begin{remark}\label{rmk:trans_3}
     It is possible to find \texttt{Fer} objects for all transpositions of the form $(y_0,x)$ in $T_k$, when $k\leq4$, of length 3 or less.
 \end{remark}
  
  Thus, in general, we have the following result.

 \begin{lemma}\label{prop:ell}
     Let $\ell(k)$ be the maximum length of a \texttt{Fer} object produced by \texttt{RecursiveFer}($x$, $k$). Then, for $k\geq3$, $$\ell(k)\leq k^2-3k+3.$$
 \end{lemma}

 \begin{proof}
    Write $R=$ \texttt{RecursiveFer}, for short. First, we argue that, for $k\geq 4$, $\len R(x_0^{(k)},k)\leq2k+5$ by induction on $k$. For simplicity, let us pick $x_0^{(k)}=c_k+1$, which corresponds to the labeling we give in Theorem~\ref{thm:Tk_Local}. We have that $\len R(c_4+1,4)\leq3$ by Remark~\ref{rmk:trans_3}. Now suppose that $k>4$. According to case $x\in C$ in the algorithm and the fact that $\rho_k^{-1}(c_k+1)=a_k+1=c_{k-1}+1$ (see Figure~\ref{general_diag}), $$R(c_k+1,k)=\rho_k*R(\rho_k^{-1}(c_k+1),k-1)*\rho_k=\rho_k*R(c_{k-1}+1,k-1)*\rho_k$$ has length at most $2(k-5)+3+2=2(k-4)+3$, as desired.
    
    Next, consider the recurrence $r(k)=4k-10+r(k-2)$, $r(3)=r(4)=3$. Inductively, $r(k)-r(k-1)\geq4$ for all $k\geq5$. Indeed, since $r(5)=20-10+r(3)=13$, we have that $r(5)-r(4)=13-3=10\geq4$ and $r(6)-r(5)=4\cdot6-10+3-13=4$. Then,
    \begin{align*}
        r(k)-r(k-1)&=4k-10+r(k-2)-4(k-1)+10-r(k-3)\\
       &=r(k-2)-r(k-3)+4>4.
    \end{align*}
    
    Finally, we use induction on $k$ to prove that $\ell(k)\leq r(k)$ for all $k\geq3$. By Remark~\ref{rmk:trans_3}, we know that $\ell(3)=\ell(4)=3=r(3)=r(4)$.
    
    Let us study the computation of $R(x,k)$. Following Algorithm~\ref{alg_tk}. In the case when $x\in A\cup B\setminus\{b\}$, $\len R(x,k)=\len R(x,k-1)\leq \ell(k-1)\leq r(k-1)<r(k)$. If $x\in C\cup\{b\}$, since $\varphi$ and $\rho$ consist of a single edge-replacement, $\len R(x,k)\leq \ell(k-1)+2\leq r(k-1)+2<r(k)$. Finally, if $x\in D$, we need to do two sub-computations: $R(c_k+1,k)$ and $R(x-c_k,k-2)$.

 Using the first claim of this proof, $$\len R(x, k)\leq 2\len R(c_k+1,k)+\ell(k-2)\leq2(2k+5)+r(k-2)=r(k).$$ This concludes the induction.

 Now, it remains to find an explicit formula for $r(k)$. But this a linear recurrence with constant coefficients of order 1, which is easily solved using the theory of eigenvectors. Namely, $r(k)=k^2-3k+1-2(-1)^k\leq k^2-3k+3$, from where the result follows.
\end{proof}

 \begin{theorem}
     The running time of Algorithm~\ref{alg_tk} is $\Theta(n^2)$ and the space complexity is $O((k-4)(n(k)-1))=O(n\log n)$.
 \end{theorem}

\begin{proof} 
 Let $t(x,k)$ be the time needed to compute \texttt{RecursiveFer}($x$, $k$). By assumption, $t(x,k)=O(1)$ for all $x$ and for all $k\leq4$. Suppose that $c$ is the maximum of such constants. Inductively, $$t(x,k)=\begin{cases}
     t(\varphi^{-1}(x),k-1)+m(\ell(k-1)), & x=b\\
     t(x,k-1), & x\in A\cup B\setminus\{b\}\\
     t(\rho^{-1}(x),k-1)+m(\ell(k-1)), & x\in C\\
     t(\rho^{-1}(x_0),k-1)+t(x-c,k-2)+m(\ell(k-1))+m(\ell(k-2)), & x\in D.
 \end{cases}$$

 If $t(k)$ is the worst time of all values of $x$, then, succinctly, $$t(k)=t(k-1)+t(k-2)+O(k^2)=cF_k+O(k^2)=O(\phi^k),$$ where $\phi$ is the golden ratio.

 In terms of $n$, $t(k)\approx\frac{cn}2$, which joined with our previous observations, we have that the time complexity is $O(n^2)$.
 
 The best case is when the input is a transposition $(y_0,x)$ with $x\in A_4\setminus\{y_0\}$. Since the sequence of sets $A_k$ is increasing, this means \texttt{RecursiveFer}$(x,k)$ will positively evaluate $x\in A_q$ $k-4$ times and then look up the generator in the hash map of $T_4$ in constant time to produce an output of length at most 3. In other words, the algorithm takes linear time (and linear space complexity too) in $k$ for this input, or $O(\log n)$. However, if the input is a conjugate of $(0\ 1)(2\ 3)(4\ 5)\cdots(2F_k-2\ 2F_k-1)$, the algorithm will evaluate \texttt{RecursiveFer}$(x,k)$ once for all $x$.
 \end{proof}
 
 To close off, we point out that since the proof of Theorem~\ref{thm:Tk_Local} uses transpositions to generate the symmetric groups, an algorithm that factors $(0\ 1)(2\ 3)(4\ 5)\cdots(2F_k-2\ 2F_k-1)$ more efficiently than Algorithm~\ref{alg_tk} likely uses a very different approach (see open questions of Section~\ref{sec:conclusion}).

 Some technical improvements may stem from the following two observations. \begin{itemize}
     \item A \texttt{Fer} sequence can be simplified in some cases. For instance, the sequence $(1\ 2\to0\ 1)(2\ 3\to1\ 2)$ is equivalent to $(2\ 3\to0\ 1)$. The algorithm produces long sequences, and thus simplification may become indispensable for practical purposes. A simple version of this is implemented in \cite{github}.
     \item Ideally, automorphisms should correspond to the trivial edge replacement $(\emptyset\to\emptyset)$, but our algorithm does not account for this. It is not immediately evident when a product of generators is an automorphism.
 \end{itemize}

\section{Conclusion and open problems}\label{sec:conclusion}

We finish this paper by stating some open problems that are left for future research concerning the topic of amoebas. It would be interesting to explore if there is a way to characterize global and local amoebas via constructions. However, this problem seems to be quite ambitious. The problems we state here concern the construction of amoebas that try to go in that direction.

\begin{enumerate}

\item Generalize the construction of~\Cref{thm:recursion}.

\item Look for other general ways of constructing global or local amoebas.

\item Characterize local amoeba trees.

\item Characterize global amoeba trees.

\item Study the balancing number of other amoeba families.

\end{enumerate}

\end{document}